\documentclass[11pt]{article}

\usepackage{a4wide}

\usepackage{amssymb}
\usepackage{amsmath}

\usepackage{datetime2}

\usepackage[latin1]{inputenc}

\usepackage{graphicx}

\usepackage{stmaryrd}

\newcommand\bigforall{\mbox{\LARGE $\mathsurround0pt\forall$}}

\usepackage{tikz}
\usetikzlibrary{arrows,shapes.gates.logic.US,shapes.gates.logic.IEC,calc}

\usepackage{xy}
\xyoption{import}

\usepackage{hyperref}

\newcommand{\ket}[1]                 	{{|#1\rangle}}

\newcommand{\Proj}               		{{\mathrm{Proj}}}

\newcommand{\cH}           		{\mathcal{H}}

\newcommand{\cD}           		{\mathcal{D}}

\newcommand{\fB}           			{\mathfrak{B}}

\newcommand{\PLQO}               	{{\mathsf{PLQO}}}

\newcommand{\CPL}                		{{\mathsf{CPL}}}

\newcommand{\RCOF}               	{{\mathsf{RCOF}}}
\newcommand{\inte}[1]                 	{{\mathop{\textstyle\int\!{#1}}}}

\newcommand{\mycirc}               	{\mathbin{@}}

\newcommand{\dummy}                {{\text{ }}}
\newcommand{\ds}                   {\displaystyle}

\newcommand{\qurtains}             {\hfill{QED}}
\newcommand{\curtains}             {\hfill{$\nabla$}}

\ifx \theorem \undefined
\newtheorem{definition}{\vspace{1mm}Definition}[section]

\newtheorem{example}[definition]{\vspace{1mm}Example}

\newtheorem{pth}[definition]{\vspace{1mm}Theorem}

\newtheorem{prop}[definition]{\vspace{1mm}Proposition}

\newenvironment{proof}{\begin{trivlist}\item{\bf Proof:}}{\qurtains\end{trivlist}}

\newcommand{\Prob}					{\textsf{Prob}}

\newcommand{\lrule}[2] 				{\frac{\ \ #1\ \ }{\ \ #2\ \ }}

\newcommand{\pnats}                 {{\mathbb{N}^+}}
\newcommand{\nats}                  {{\mathbb N}}

\newcommand{\reals}                 {{\mathbb R}}
\newcommand{\comps}                 {{\mathbb C}}

\newcommand{\id}                    {{\textsf{id}}}

\newcommand{\ovl}[1]                {{\overline{#1}}}
\newcommand{\udl}[1]                {{\underline{#1}}}

\newcommand{\dom}                   {\mathop{\textsf{dom}}}

\newcommand{\comp}[1]               {{#1^{\textsf{c}}}}

\newcommand{\lrt}[1]                {{\langle#1\rangle}}

\newcommand{\inv}[1]                {{{#1}^{-1}}}

\newcommand{\da}                    {{{\downarrow}}}
\newcommand{\ua}                    {{{\uparrow}}}

\newcommand{\from}                  {{{\leftarrow}}}

\newcommand{\lverum}                {{\mathsf{t\!t}}}
\newcommand{\lfalsum}               {{\mathsf{f\!f}}}

\newcommand{\lM}                  	{\mathop{{\mathsf M \!}}}
\newcommand{\lO}                  	{\mathop{{\odot}}}

\newcommand{\lneg}                  {\mathop{\neg}}

\newcommand{\limp}                  {\mathbin{\supset}}
\newcommand{\leqv}                  {\mathbin{\equiv}}
\newcommand{\lconj}                 {\mathbin{\wedge}}
\newcommand{\ldisj}                 {\mathbin{\vee}}

\newcommand {\num}[1]              {{#1}}

\newcommand{\satfo}                  {\mathbin{\,\sat_{\mathsf{fo}}\,}}
\newcommand{\nsatfo}                 {\mathbin{\,\not\sat_{\mathsf{fo}}\,}}
\newcommand{\satc}                    {\sat_{\mathsf{c}}}
\newcommand{\entc}                    {\ent_{\mathsf{c}}}
\newcommand{\sat}                    {\Vdash}
\newcommand{\ent}                    {\vDash}
\newcommand{\der}                    {\vdash}

\newcommand{\nmycirc}               	{\mathbin{\not \!@}}

\newcommand {\Lc}                   {{L_{\mathsf{c}}}}

\newcommand {\HYP}		{\text{HYP}}
\newcommand {\TT}          	{\text{TT}}
\newcommand {\MP}            	{\text{MP}}

\newcommand {\RR}          	{\text{RR}}

\newcommand{\NC}               	{{\mathsf{NC}}}

\newcommand {\eB}        	       {{e\hspace*{-0.3mm}B}}

\begin{document}

\title{Probabilistic logic of quantum observations} 

\author{A.~Sernadas${}^1 \dag$\ \ J.~Rasga${}^1$\ \ C.~Sernadas${}^1$\ \ 
							L.~Alcácer${}^2$\ \ A.~B.~Henriques${}^3$\\[1.5mm]
{\scriptsize ${}^1$ Dep.~Matemática, Instituto Superior Técnico 
 and CMAF-CIO}\\[0.5mm]
{\scriptsize ${}^2$ Dep.~de Engenharia Química, Instituto Superior Técnico and 
Instituto de Telecomunicações}\\[0.5mm]
{\scriptsize ${}^3$ Dep.~Física, Instituto Superior Técnico and Centro Multidisciplinar de Astrofísica}\\[0mm]
{\scriptsize Universidade de Lisboa, Portugal}}

\date{March, 2018}

\maketitle

\begin{abstract} 
A probabilistic propositional logic, endowed with an epistemic component for asserting (non-)compatibility of diagonizable and bounded observables, is presented and illustrated for reasoning about the random results of projective measurements made on a given quantum state.~Simultaneous measurements are assumed to 
imply that the underlying observables are compatible.~A sound and weakly complete axiomatization is provided relying on the decidable first-order theory of real closed ordered fields.  The proposed logic is proved to be a conservative extension of classical propositional logic.
\\[2mm]
{\bf Keywords}: quantum logic, probabilistic logic, epistemic logic.\\[2mm]
{\bf AMS MSC2010}: 03G12, 81P10, 03B48.
\end{abstract}


\section{Introduction}\label{sec:introduction} 

The origin of quantum logic can be traced back to Garrett Birkhoff and John von Neumann's 1936 paper (see~\cite{neu:36}). Since then,  the question of reasoning in logic about quantum phenomena has been a recurrent research topic.  

The introduction of quantum mechanics concepts in formal logic is quite challenging since there is the need to accommodate the continuous nature of quantum mechanics within the discrete setting of symbolic reasoning. 

There are two main approaches to quantum logic. In the original approach, proposed in the seminal paper~\cite{neu:36}, each propositional symbol is an experimental proposition over a set of compatible observables, that is, a subset of the set of possible results of measuring them jointly. The denotation of such a propositional symbol is a closed linear subspace of the Hilbert space at hand. 
 The connectives of the logic reflect the operations in the lattice of all closed subspaces of a Hilbert space. As a consequence, some classical laws like distributivity of $\lconj$ and $\ldisj$ are no longer valid since the observables associated with the experimental propositions in those laws may not be necessarily compatible. If they were all compatible then the classical laws for $\lconj$ and $\ldisj$ would be true (see page 831 of~\cite{neu:36}).
There have been many investigations on quantum logic based on this approach (see, for instance,~\cite{har:81,tak:81,chi:02,pav:08,har:16,pav:16}), including   first-order frames and quantum set theory. 

The other approach, assumes {\it à priori}, that all observables involved are compatible with each other (see~\cite{mey:03,pmat:acs:05}) and focus on extensions of classical propositional logics  with new 
constructions for reasoning about particular quantum mechanics notions (inspired by the way probabilistic reasoning was introduced in classical logic~\cite{hal:03,acs:jfr:css:15}). For an overview of quantum logic according to this approach  and quantum structures, see~\cite{gab:07,gab:09}.

Herein, we address the problem of reasoning about the random results of projective measurements made on a given quantum state. Since the observables involved in the measurements are not necessarily compatible
we need also to reason about incompatibility of observables. 
This was achieved by defining an  epistemic-like 
operator $\lO \alpha$ with the intended meaning {\em it is possible to know the truth value of propositional formula} $\alpha$ (that is, the observables associated with the propositional symbols occurring in $\alpha$ are compatible). 

More precisely, we develop a logic ($\PLQO$) where each propositional symbol is an assertion on the result of 
a measurement (mathematically described by an observable) and a  formula is a Boolean assertion on the possible results of the measurements associated with its propositional symbols. The probability of such a formula $\alpha$ being true is expressed in our logic by $\inte{\alpha}$. 
The present paper generalizes the work in~\cite{pmat:acs:05} namely by presenting 
a more general semantics as well as a different calculus coping with both the probabilities and incompatibility
of observables. 
 
We note that the Boolean connectives inside a formula $\alpha$ have the same meaning of the connectives of the logic in~\cite{neu:36}.
  However, in our logic $\alpha$ only appears in the scope of either $\inte{\alpha}$ 
or in the scope of $\lO \alpha$.  When the observables in $\alpha$ are not compatible, the Boolean connectives in that formula do not play any role since  we state that  $\inte{\alpha} \,@ \, r$ is false. The same applies to $\lO \alpha$ since the evaluation of this formula corresponds only to checking whether or not the observables in $\alpha$ are compatible.

As a consequence, our logic is a conservative extension of classical propositional logic in the sense that we were able to prove that 
$$(\lO \alpha) \limp (\inte{\alpha} =1) \text{ is valid in  }  \PLQO\quad \text{ iff } \quad \alpha \text{ is classically valid}.$$
That is, $\alpha$ is always true iff $\alpha$ holds in every state of a quantum system whenever
the underlying observables are compatible. 

As in \cite{neu:36} (page 826), 
 we assume that each finite set of compatible observables is simultaneously diagonalizable, that is, there exists an orthonormal basis of the underlying Hilbert space composed by common eigenvectors of the observables. This is the case, for instance,  when the observables are defined by pairwise commuting compact self-adjoint operators (see~\cite{ish:95,bla:08,hal:13}).


In Section~\ref{sec:preliminaries}, we briefly review some relevant mathematical concepts underlying quantum mechanics. Moreover, we define our context namely by assuming, for the purpose of this work, that each observable is diagonizable and bounded besides being self-adjoint. With the aim of setting-up the probability context for joint measurements of compatible observables we define the observable resulting from the product 
of two given compatible observables. In Section~\ref{sec:qvars}, we
 introduce the notion of propositional quantum variable and analyze compatibility of propositional quantum variables. 
 
After that, in Section~\ref{sec:plqo}, we introduce
the language of $\PLQO$ after  recalling some notions of classical propositional logic as well as of the theory of real closed ordered fields. 
This language has two kinds of atomic formulas: $\lO \alpha$ as referred above,  and $\inte{\alpha} \mycirc p$ with the intended meaning of {\em it is possible to know the truth value of $\alpha$} and
 {\it the probability of $\alpha$ being true is $\mycirc p$} where $\mycirc$ is a relational operator.
Then, we introduce the notion of quantum interpretation structure  which is  composed by a Hilbert space representing the envisaged  quantum system, a unit vector representing the quantum state and the interpretation of each propositional symbol as a propositional quantum variable. Satisfaction and entailment are then defined. The section ends by presenting an axiomatization for $\PLQO$ capitalizing on the decidability of the theory of real closed ordered fields. This axiomatization is shown to be strongly sound in Section~\ref{sec:ssound} after establishing some auxiliary results. 

In Section~\ref{sec:wcomp}, we show that the axiomatization is weakly complete starting by proving a finite-dimensional model existence lemma relating satisfaction  of a formula in $\PLQO$ with satisfaction of the translation of the formula in
the $\RCOF$ interpretation structure having as domain the set $\reals$. Strong completeness is out of question since $\PLQO$ is not a compact logic as we explain in Section~\ref{sec:plqo}. Conservativeness of $\PLQO$ with respect to classical propositional logic is discussed in Section~\ref{sec:cons}. Finally, in Section~\ref{sec:epistemic}, we investigate the epistemic nature of observability and provide several examples. We conclude the paper in Section~\ref{sec:concrems} with the global assessment of what was done and future research directions.

\section{Preliminaries}\label{sec:preliminaries}

Our objective here is to provide a very brief account 
of what we need to know about 
the random results of projective measurements made on a quantum state.
For a full account see~\cite{rud:87,rud:91,ish:95,bla:08,hal:13}.

\subsection*{Basic notions}

Following  the mathematical formulation of quantum mechanics, as  introduced by Dirac in 1930, \cite{dir:67}, and by von Neumann in 1932, \cite{neu:32}, a quantum system is defined by a complex separable Hilbert space $\cH$. Moreover, a quantum state of the system is represented by a unit vector in $\cH$.

We start by reviewing the concept of projection-valued measure. 
Given a set $\Omega$ and a $\sigma$-algebra $\fB$ over $\Omega$, 
a {\it projection-valued measure} $\mu$ is a map from $\fB$ to the class of bounded operators over $\cH$ having the following properties.
\begin{itemize}

\item For each $S \in \fB$, $\mu(S)$ is an orthogonal projection, that is, $\mu(S)^2=\mu(S)$ and
$\mu(S)$ is a self-ajoint operator.

\item $\mu(\emptyset) = 0$ and $\mu(\Omega)= I_\cH$, where $I_\cH$ is the identity over $\cH$.
\item Given a countable family $\{S_j\}_{j\in J}$ of pairwise disjoint elements of $\fB$, let
	$$S = \bigcup_{j\in J}S_j.$$
	Then, for every $\psi \in \cH$,
	$$\mu(S) \psi= \sum_{j\in J} \mu(S_j)\psi,$$
	where the convergence of the sum is in the norm topology on $\cH$, whenever $J$ is countably infinite.
\item For every $S_1,S_2 \in \fB$,
	$$\mu(S_1 \cap S_2) = \mu(S_1) \mu(S_2).$$
\end{itemize}
Any projection-valued measure $\mu$ and $\psi \in \cH$ induce a positive real-valued measure
$$\mu_\psi: S \mapsto \langle \psi,\mu(S)\psi\rangle$$
 over $\fB$. 
Moreover, $\mu$ induces a unique linear map $$f \mapsto \int_\Omega f d\mu$$ 
that maps each measurable and complex function to an operator on $\cH$ such that
$$\left\langle \psi,\left( \int_\Omega f \; d\mu\right)\psi\right\rangle = \int_\Omega f \; d\mu_\psi,$$
for all $\psi$  in 
$$\left\{\psi \in \cH: \int_{\Omega} |f(\lambda)|^2 d\mu_\psi(\lambda) < \infty\right\}.$$
This operator satisfies the following properties:
\begin{itemize}
\item $\int_\Omega 1_S\; d\mu =\mu(S)$ for every $S \in \fB$;
\item $\int_\Omega fg \;d\mu=\left(\int_\Omega f \; d\mu\right) \left(\int_\Omega g \; d\mu\right)$.
\end{itemize}

A projective measurement  is defined by a self-adjoint operator (aka {\it observable}) 
acting on $\cH$.
The possible results of the measurement are the values in the spectrum $\Omega$ of the observable.
Clearly, $\Omega\subseteq\reals$ since the operator is self-adjoint.
From now on, let $\fB$ be the $\sigma$-field of the Borel subsets of $\Omega$ induced by the usual topology of $\reals$.

The spectral theorem tells us that a self-adjoint operator  $O$ induces a unique projection-valued measure $\mu$
 such that
$$\int_\Omega \id_\Omega \, d\mu = O,$$
which is called the {\it spectral representation} of $O$, where $\id_\Omega$ is the identity over $\Omega$.
Observe that we are not imposing that $O$ is a bounded operator. In the sequel, we denote by $\dom O \subseteq \cH$ the domain of  $O$.

The following definition is used in the sequel. If $O$ is a self-adjoint operator and
$f:\Omega \to \comps$ then  the {\it image} of $O$ by $f$ is the operator
$$(\ddag) \qquad f(O)=\int_\Omega f \; d\mu$$
where $\mu$ is the projection-valued measure induced by $O$.

\subsection*{Observables}

The notion of propositional quantum variable, introduced below,  assumes that an observable besides being self-adjoint is also diagonizable. 
We say that an operator $O:\cH \to \cH$ is {\it diagonizable} whenever 
\begin{itemize}

\item $O$ is a self-adjoint  operator;

\item $O$ has a countable pure point spectrum $\sigma_p(O)$;

\item $\{\Proj_\omega\}_{\omega \in \sigma_p(O)}$ is a complete system of eigenprojections of $O$, that is, 

\begin{itemize}
\item $\{\psi \in \dom O: \Proj_\omega \psi =\psi\}$ is the eigenspace of $O$ associated with the eigenvalue $\omega$, for each $\omega \in \sigma_p(O)$;

\item $\displaystyle \sum_{\omega \in \sigma_p(O)} \Proj_\omega=I_\cH$. 
\end{itemize}
\end{itemize}
Then, as shown in~\cite{bla:08},  
$$\int_\Omega \id_\Omega \, d\mu = O$$
where $\mu$ is the discrete projection-valued measure such that
$$\mu(S) = \sum_{\omega \in S \cap \sigma_p(O)} \chi_S(\omega) \Proj_\omega$$
for every $S \in \fB$. Moreover,
$$O\psi =\sum_{\omega \in \sigma_p(O)} \omega \Proj_\omega \psi$$
for every $\psi$ in the domain of $O$. That is, $O$ is a suboperator of $\sum_{\omega \in \Omega} \omega \Proj_\omega$. When $O$ is a bounded operator then 
$$O =\sum_{\omega \in \sigma_p(O)} \omega \Proj_\omega.$$
For example, all compact operators are diagonizable (see Riesz-Schauder Theorem, see~\cite{bla:08}, pp 79). Moreover, there are operators that although not compact are still diagonizable (see~\cite{hal:13} pp 124) as the Hamiltonian of the simple harmonic oscillator whose eigenvalues are $(n+\frac 12)\hbar \omega$ with $n \in \nats$ (see~\cite{ish:95} pp 60 and \cite{bla:08} pp 262). 

From now on, we assume that all observables are diagonizable. Moreover, in order to avoid the problem
of domains in unbounded operators, we also assume that all observables are bounded.

\begin{example} \em \label{ex:spin1}
Consider a spin-$\frac 12$ particle (see~\cite{dir:67,gri:14}). The Hilbert space for the spin is $\cH=\comps^2$.
Let $O_x,O_y,O_z: \cH \to \cH$ be the observables for the $x$, the $y$ and the $z$ components of the spin defined 
by the matrices:
$$O_x = \frac 12 \hbar \left[\begin{array}{cc}
0 & 1\\
1 & 0
\end{array}\right]\quad  O_y = \frac 12 \hbar \left[\begin{array}{cc}
0 & -i\\
i& \ \  0
\end{array}\right] \quad O_z=\frac 12 \hbar \left[\begin{array}{cc}
1 & \ \ 0\\
0& -1
\end{array}\right],$$
respectively. 
The  set of eigenvalues is $\{\frac \hbar 2, -\frac \hbar 2\}$ for the three operators. Moreover,
the orthonormal bases induced by $O_x$, $O_y$ and $O_z$ for $\comps^2$ are:
$$\left\{\frac {1}{\sqrt 2}  \left[\begin{array}{c}
1\\
1
\end{array}\right], \frac {1}{\sqrt 2} \left[\begin{array}{c}
\ 1\\
-1
\end{array}\right]\right\},\;  \left\{\frac {1}{\sqrt 2} \left[\begin{array}{c}
1\\
i
\end{array}\right], \frac {1}{\sqrt 2}  \left[\begin{array}{c}
\ 1\\
-i
\end{array}\right]\right\},\;\left\{\left[\begin{array}{c}
1\\
0
\end{array}\right],  \left[\begin{array}{c}
0\\
1
\end{array}\right]\right\},$$
respectively. 
So $O_x$, $O_y$ and $O_z$ are diagonalizable observables. For example, the spectral representation of $O_x$ is as follows:
$$O_x = \frac {1}{\sqrt 2}  \left[\begin{array}{c}
1\\
1
\end{array}\right] \frac \hbar 2 \left \langle\frac {1}{\sqrt 2}  \left[\begin{array}{c}
1\\
1
\end{array}\right] \, ,\, \cdot\,  \right \rangle -  \frac {1}{\sqrt 2}  \left[\begin{array}{c}
\ 1\\
-1
\end{array}\right]\frac \hbar 2 \left \langle\frac {1}{\sqrt 2}  \left[\begin{array}{c}
\ 1\\
- 1
\end{array}\right]\, , \, \cdot \, \right \rangle.$$
It can be easily seen that no pair of these operators commute. However,  the  total angular momentum $O^2$ defined as
$$O^2= O_x^2 + O_y^2 + O_z^2=\frac 34 \hbar \left[\begin{array}{cc}
1 & 0\\
0 & 1
\end{array}\right]$$
is such that 
$$O^2 O_x=O_x O^2 \quad O^2 O_y=O_y O^2 \quad O^2 O_z=O_z O^2.$$
Thus, $O^2$ and $O_x$, $O^2$ and $O_y$, and $O^2$ and $O_z$ commute and so, as we shall see later on, they are compatible observables.
Observe that $O^2$ is diagonizable as well.\curtains
\end{example}

As we shall see later on each observable and unit vector induce a probability space allowing the calculation
of the probabilities of obtaining a certain result after a measurement of that observable. When performing simultaneous measurements by a finite set of pairwise commuting observables we need to know the probability space where we can calculate the probabilities of the joint results.  This probability space can be induced by
the observable resulting from the product of the observables at hand  (see~\cite{ish:95}).

We now define the product of observables starting by introducing the notion of morphism between  observables.
Let $O_1$ and $O_2$ be observables and $\{\Proj_\omega^1\}_{\omega \in \sigma_p(O_1)}$ be a complete system of eigenprojections of $O_1$. A morphism $$f: O_1 \to O_2$$ is a map
$f: \sigma_p(O_1) \to \comps$ such that $f(O_1)=O_2$, where $f(O_1)$ is such that 
$$\displaystyle f(O_1)=\sum_{\omega \in \sigma_p(O_1)} f(\omega) \Proj_{\omega}.$$

\begin{prop} \em \label{prop:product}
Let $O'$ and $O''$ be a  pair of (diagonizable and bounded) commuting observables such that 
$$\displaystyle O'=\sum_{\omega' \in \sigma_p(O')} \omega' \Proj'_{\omega'} \quad \text{and} \quad O''=\sum_{\omega'' \in \sigma_p(O'')} \omega'' \Proj''_{\omega''},$$
where 
$\{\Proj'_{\omega'}\}_{\omega' \in \sigma_p(O')}$ and $\{\Proj''_{\omega''}\}_{\omega'' \in \sigma_p(O'')}$ are  complete systems of eigenprojections of $O'$ and $O''$, respectively.
Let
$$O' \times O''=\displaystyle \sum_{(\omega',\omega'') \in\sigma_p(O') \times \sigma_p(O'')} \lambda_{\omega'\omega''} \,\Proj'_{\omega'} \Proj''_{\omega''}$$
be an operator where $\lambda_{\omega'\omega''}\in (0,1)$ such that $\lambda_{\omega'_1\omega''_1} \neq \lambda_{\omega'_2\omega''_2}$ for every $\omega'_1,\omega'_2 \in \sigma_p(O')$ and $\omega''_1,\omega''_2 \in \sigma_p(O'')$ with
$(\omega'_1,\omega''_1) \neq (\omega'_2,\omega''_2)$.
Then, $$O' \times O''$$ is a diagonizable and bounded operator such that there are 
 morphisms $f': O' \times O'' \to O'$ and $f'':O' \times O'' \to O''$.
\end{prop}
\begin{proof}\ \\
We start by proving some auxiliary results.\\[1mm] 
(a)~$\Proj'_{\omega'}\Proj''_{\omega''}=\Proj''_{\omega''}\Proj'_{\omega'}$, for every $\omega' \in \sigma_p(O')$ and 
$\omega'' \in \sigma_p(O'')$.
Observe that, for every $\psi \in \cH$,
$$O'O''\Proj'_{\omega'}\psi = O''O'\Proj'_{\omega'}\psi=\omega' (O''\Proj'_{\omega'}\psi)$$
since $O'$ and $O''$ commute. 
Hence,  $O''\Proj'_{\omega'}\psi$ belongs to the eigenspace of $O'$ associated with $\omega'$.
Therefore,
$$\Proj'_{\omega'}O''\Proj'_{\omega'}=O''\Proj'_{\omega'}.$$
Moreover, 
$$\Proj'_{\omega'}O''\Proj'_{\omega'}=\Proj'_{\omega'}O''.$$
Therefore, 
$$O''\Proj'_{\omega'}=\Proj'_{\omega'} O''.$$
Similarly,  $$O'\Proj''_{\omega''}=\Proj''_{\omega''} O'.$$
Thus, 
$$O'' \Proj'_{\omega'}\Proj''_{\omega''}\psi = \Proj'_{\omega'} O''  \Proj''_{\omega''}\psi=
\omega'' \Proj'_{\omega'}  \Proj''_{\omega''}\psi$$
and so $\Proj'_{\omega'} \Proj''_{\omega''}\psi$ belongs to the eigenspace of $O''$ associated with $\omega''$. Hence,
$$\Proj''_{\omega''} \Proj'_{\omega'} \Proj''_{\omega''}\psi= \Proj'_{\omega'} \Proj''_{\omega''}\psi.$$
Moreover,
$$\Proj''_{\omega''} \Proj'_{\omega'} \Proj''_{\omega''}\psi=  \Proj''_{\omega''}\Proj'_{\omega'}\psi.$$
Then,
$$\Proj'_{\omega'}\Proj''_{\omega''}=\Proj''_{\omega''}\Proj'_{\omega'}.$$
(b)~$\Proj'_{\omega'}\Proj''_{\omega''}$ is a projection for every $\omega' \in \sigma_p(O')$ and 
$\omega'' \in \sigma_p(O'')$. In fact, 
$$(\Proj'_{\omega'}\Proj''_{\omega''})^2=\Proj'_{\omega'}\Proj''_{\omega''},$$  using  (a) 
and the fact   $(\Proj'_{\omega'})^2 = \Proj'_{\omega'}$ and $(\Proj''_{\omega''})^2=\Proj''_{\omega''}$. Moreover, 
$$(\Proj'_{\omega'}\Proj''_{\omega''})^*=\Proj'_{\omega'}\Proj''_{\omega''}$$
also by (a) and the fact  
$(\Proj'_{\omega'})^*=\Proj'_{\omega'}$ and 
$(\Proj''_{\omega''})^*=\Proj''_{\omega''}$.
\\[2mm]
We are now ready to prove the envisaged properties of $O' \times O''$.\\[1mm]
(1)~$O' \times O''$ is a self-adjoint operator. Observe that:
$$(O' \times O'')^*=\displaystyle \sum_{(\omega',\omega'') \in\sigma_p(O') \times \sigma_p(O'')} \lambda_{\omega'\omega''} \,(\Proj'_{\omega'} \Proj''_{\omega''})^*.$$
Note that,
$$(\Proj'_{\omega'} \Proj''_{\omega''})^*=\Proj'_{\omega'} \Proj''_{\omega''}$$
since $\Proj'_{\omega'} \Proj''_{\omega''}$ is a projection by (b). Therefore,
$$(O' \times O'')^*=O' \times O''.$$
(2)~$O' \times O''$ is diagonizable.\\[1mm]
(a)~It is immediate to see that the spectrum of $O' \times O''$ is countable.
\\[1mm]
(b)~$\{\Proj'_{\omega'}\Proj''_{\omega''}\psi: \psi \in \cH\}=\{\psi \in \cH: O' \times O''\psi=\lambda_{\omega'\omega''}\psi\}$, for each $\omega' \in \sigma_p(O')$ and $\omega'' \in \sigma_p(O'')$ such that
$\Proj'_{\omega'}\Proj''_{\omega''} \neq 0$.
That is, $$\{\Proj'_{\omega'}\Proj''_{\omega''}\psi: \psi \in \cH\}$$ is an eigenspace of $O' \times O''$ for the eigenvalue
$\lambda_{\omega'\omega''}$.
Let $\omega'_1 \in \sigma_p(O')$ and $\omega''_1 \in \sigma_p(O'')$ such that
$\Proj_{\omega'_1}\Proj_{\omega''_1} \neq 0$.\\[1mm]
$(\subseteq)$ Assume that $\Proj_{\omega'_1}\Proj_{\omega''_1}\psi \in \{\Proj_{\omega'_1}\Proj_{\omega''_1}\psi: \psi \in \cH\}$. Observe that, 
$$\begin{array}{l}
O' \times O''\Proj_{\omega'_1}\Proj_{\omega''_1}\psi =\\[5mm]
  \hspace*{20mm}  = \displaystyle \sum_{(\omega',\omega'') \in\sigma_p(O') \times \sigma_p(O'')} \lambda_{\omega'\omega''} \,\Proj'_{\omega'} \Proj_{\omega'_1} \Proj''_{\omega''}\Proj_{\omega''_1} \psi\\[7mm]
 \hspace*{20mm} = \lambda_{\omega'_1\omega''_1} \,\Proj_{\omega'_1}\Proj_{\omega''_1} \psi.
 \end{array}
 $$
 $(\supseteq)$ Assume that $\psi \in \{\psi \in \cH: O' \times O''\psi=\lambda_{\omega'_1\omega''_1}\psi\}$, that is, $\psi$ is such that
 $$O' \times O''\psi=\lambda_{\omega'_1\omega''_1}\psi.$$
 That is, 
 $$\displaystyle \left(\sum_{(\omega',\omega'') \in\sigma_p(O') \times \sigma_p(O'')} \lambda_{\omega'\omega''} \,\Proj'_{\omega'} \Proj''_{\omega''}\right)\psi=\lambda_{\omega'_1\omega''_1} \psi.$$
 Thus, $$\psi=\Proj_{\omega'_1} \Proj_{\omega''_1}\psi.$$
 Therefore, $$\psi \in \{\Proj_{\omega'_1}\Proj_{\omega''_1}\psi: \psi \in \cH\}.$$
(c)~$\Proj'_{\omega'}\Proj''_{\omega''}\psi$ is a simultaneous eigenvector of $O'$ and $O''$ with respect to $\omega'$ and $\omega''$, for each $\psi \in \cH$, $\omega' \in \sigma_p(O')$ and $\omega'' \in \sigma_p(O'')$ such that
$\Proj'_{\omega'}\Proj''_{\omega''} \neq 0$.
Indeed, let $\psi \in \cH$, $\omega'_1 \in \sigma_p(O')$ and $\omega''_1 \in \sigma_p(O'')$ such that
$\Proj_{\omega'_1}\Proj_{\omega''_1} \neq 0$. Then,
$$\begin{array}{rcl}
O'\Proj_{\omega'_1}\Proj_{\omega''_1}\psi &=&
\displaystyle \sum_{\omega' \in\sigma_p(O')} \omega'\,\Proj'_{\omega'} \Proj_{\omega'_1}\Proj_{\omega''_1} \psi\\[5mm]
 &=&  \omega'_1 \,\Proj_{\omega'_1}\Proj_{\omega''_1} \psi.
 \end{array}
 $$
(d)~There is a countable orthonormal basis of $\cH$ composed by eigenvectors of $O' \times O''$ which are common eigenvectors of both $O'$ and $O''$. 
Observe that
$$\dim \cH=\displaystyle 
		\sum_{\text{\shortstack[c]
{$(\omega', \omega'')\in \sigma_p(O') \times \sigma_p(O'')$\\
$\Proj'_{\omega'}\Proj''_{\omega''} \neq 0$
}
}} \dim\{\psi \in \cH: \Proj'_{\omega'} \Proj''_{\omega''}\psi\}
$$
by (b) and since $\dim\{\psi \in \cH: \Proj'_{\omega'} \Proj''_{\omega''}\psi\}=0$ whenever 
$\Proj'_{\omega'}\Proj''_{\omega''} = 0$.
Then, an orthonormal basis for $\cH$  is obtained  by joining the orthonormal bases of each eigenspace of $O' \times O''$.
Moreover, each vector in this basis is also an eigenvector of $O'$ and $O''$ as shown in (c). \\[1mm]
%
(3)~There are morphisms $f': O' \times O'' \to O'$ and $f'':O' \times O'' \to O''$. 
In fact let $f': \sigma_p(O' \times O'') \to \comps$ and $f'':  \sigma_p(O' \times O'') \to \comps$ be maps such that 
$f'(\lambda_{\omega'\omega''})=\omega'$ and $f''(\lambda_{\omega'\omega''})=\omega''$. 
We start by observing that
$$\sum_{\omega'' \in \sigma_p(O'')} \Proj''_{\omega''}=I_{\cH}.$$
Hence, 
$$
\begin{array}{rcl}
f'(O' \times O'') &=& \displaystyle \sum_{(\omega',\omega'') \in \sigma_p(O') \times \sigma_p(O'')} f'(\lambda_{\omega'\omega''}) \Proj'_{\omega'} \Proj''_{\omega''}\\[5mm]
&=& \displaystyle \sum_{\omega' \in \sigma_p(O')} \omega' \Proj'_{\omega'} \sum_{\omega'' \in \sigma_p(O'')}  \Proj''_{\omega''}\\[5mm]
&=& \displaystyle \sum_{\omega' \in \sigma_p(O')} \omega' \Proj'_{\omega'} I_\cH\\[5mm]
&=& O'.
\end{array}
$$
Similarly, $f''(O' \times O'')=O''$.\\[2mm]
(4)~$O' \times O''$ is a bounded operator. Indeed, its domain is $\cH$ and it is self-adjoint so it is bounded (see~\cite{hal:13} Corollary 9.9 in page 173).
\end{proof}

We can say that $O' \times O''$ is the observable where the observables $O'$ and $O''$ can be simultaneously measured. The notation $O' \times O''$ is justified since $O' \times O''$ is a binary product
in the category of  observables and their morphisms. Indeed, the universal property also holds. Let $g': O''' \to O'$ and $g'': O''' \to O''$ be 
morphisms. Let $g: \sigma_p(O''')\to \comps$ be defined as follows:
$g(\omega''')=\lambda_{g'(\omega''')g''(\omega''')}.$
Then, it is immediate to see that $g$ is the unique morphism from $O'''$ to $O' \times O''$ such that
$f' \circ g=g'$ and $f'' \circ g=g''$.

\section{Propositional quantum variables}\label{sec:qvars}

By a {\it propositional quantum variable} we mean a triple $Y=(O,\cD,\ua)$ 
where
$O$ is an observable acting on $\cH$ and
$\cD$ is a finite partition of $\sigma_p(O)$ and $\ua \in \cD$.
We say that $Y$ is defined on $\cH$ and takes values in $\cD$. For the purpose of this paper, the  set $\ua$ plays the role of verum and  $\sigma_p(O) \setminus \ua$,  denoted by $$\da,$$ plays the role of falsum.

We want to reason about the probability of the measurement of observable $O$ having a result in either $\ua$ or $\da$ when the  system is in a state $\psi$.
In fact, an observable $O$ and a state $\psi$ induce the probability space 
$(\sigma_p(O),\wp \sigma_p(O),\mu_\psi)$ where $\mu_\psi$ is the probability measure such that 
$$\mu_\psi(\omega)=\langle \psi,\mu(S)\psi\rangle= \langle \psi,\Proj_\omega \psi\rangle,$$
where $\mu$ is the discrete projection-valued measure induced by $O$ and $S$  is a Borel set such $S \cap \sigma_p(O) =\{\omega\}$. Observe that the value of any Borel set with this property by $\mu$ is the same.
Moreover, for every $D \subseteq \sigma_p(O)$,
$$\mu_\psi(D)= \sum_{\omega \in D} \mu_\psi(\omega)=\langle \psi,\mu(S)\psi\rangle=\langle \psi,\Proj_D\psi\rangle,$$ where 
$S$  is a Borel set such $S \cap \sigma_p(O) =D$ and
$$\Proj_D=\sum_{\omega \in D}  \Proj_\omega.$$

%

Now we are ready, given a quantum variable  $Y=(O,\cD,\ua)$ and a state $\psi$,  to define a random variable 
$$Y_\psi =\omega \mapsto D_\omega: \sigma_p(O) \to \cD,$$
on $(\sigma_p(O),\wp \sigma_p(O),\mu_\psi)$, where $D_\omega$ is the element of the partition $\cD$ to which $\omega$ belongs. 
Observe that, the probability distribution induced by $Y_\psi$ on $\cD$ is such that:
$$D \mapsto \Prob\left(Y_\psi = D\right) =  \mu_\psi(D): \cD \to [0,1].$$


We next address what happens when one wants to perform several simultaneous measurements on a given state of the quantum system.

\subsection*{Compatibility of propositional quantum variables}

The theory of quantum mechanics tells us that
it is possible to measure simultaneously a finite number of observables $O_1,\dots,O_n$ 
if only if they are {\it compatible}, that is, $O_i$ and $O_j$ commute for every $i,j=1,\dots,n$ with $i \neq j$. 

\begin{example} \em \label{ex:compspin}
As seen in Example~\ref{ex:spin1}, the observables $O^2$ and $O_x$, $O^2$ and $O_y$, and $O^2$ and $O_z$ commute and so they are compatible.\curtains

\end{example}

\begin{example} \em \label{ex:hydro1} Consider the electron of a hydrogen atom. Then, its
Hamiltonian $H$, total angular momentum $L^2$, $z$-component $L_z$ of the angular momentum
and $z$-component $S_z$ of the spin operator are compatible (see~\cite{ish:95} pp 119).
Denote the state vectors of a hydrogen atom by $\psi_{n\ell m \lambda}$, where
$n >0$, $\ell =0,\dots,n-1$ and $m=-\ell,\dots,\ell+1$ are natural numbers and $\lambda =\pm \frac 12$. 
These are simultaneous eigenvectors of those observables. The eigenvalues are as follows: 
\begin{itemize}
\item $H \psi_{n\ell m \lambda}=E_n \psi_{n\ell m \lambda}$ where $E_n=-\frac{R}{n^2}$ and $R$ is the Rydberg constant (see~\cite{hal:13} pp 8);
\item $L^2 \psi_{n\ell m \lambda}= \ell(\ell+1)\hbar^2\psi_{n\ell m \lambda}$;
\item $L_z \psi_{n\ell m \lambda}= m\hbar\psi_{n\ell m \lambda}$;
\item $S_z \psi_{n\ell m \lambda}= \lambda \hbar\psi_{n\ell m \lambda}$. \curtains
\end{itemize}
\end{example}

When performing simultaneous measurements of compatible  observables $O'$ and $O''$ on a state $\psi$, the probability space where we can calculate the probabilities of the joint results is 
$$(\sigma_p(O' \times O''),\wp \sigma_p(O' \times O''),\mu_\psi)$$
 induced by $O' \times O''$. Note that,
 $$\mu_\psi(\lambda_{\omega'\omega''}) = \langle \psi,\Proj'_{\omega'} \Proj''_{\omega''}\psi\rangle.$$ 
 Moreover, for any 
 $$D_{D'D''}=\{\lambda_{\omega'\omega''}: \omega' \in D' \text{ and } \omega'' \in D''\}$$ where
 $D' \subseteq \sigma_p(O')$ and $D'' \subseteq \sigma_p(O'')$, we have

 \begin{align*}  
 \mu_\psi(D_{D'D''})=
 \langle \psi,\sum_{\lambda_{\omega'\omega''}\in D_{D'D''}}\Proj'_{\omega'} \Proj''_{\omega''}\psi\rangle
\displaybreak[1] \\[1mm]
& \hspace*{-57mm} \displaystyle 
=  \langle \psi,\sum_{\omega' \in D'} \sum_{\omega'' \in D''} \Proj'_{\omega'} \Proj''_{\omega''}\psi\rangle
\displaybreak[1]\\[1mm]
& \hspace*{-57mm} \displaystyle 
=  \langle \psi,\sum_{\omega' \in D'} \Proj'_{\omega'}  \sum_{\omega'' \in D''} \Proj''_{\omega''}\psi\rangle
\displaybreak[1]\\[1mm]
& \hspace*{-57mm} \displaystyle 
=  \langle \psi,\left(\sum_{\omega' \in D'} \Proj'_{\omega'}\right)   \Proj''_{D''}\psi\rangle
\displaybreak[1]\\[1mm]
& \hspace*{-57mm} \displaystyle 
=  \langle \psi,\Proj'_{D'}   \Proj''_{D''}\psi\rangle.
\displaybreak[1]
\end{align*}
 
 We proceed now with the task of bringing the notion of compatibility to the realm of propositional quantum variables. We say that pqvs $Y'=(O',\cD',\ua')$ and $Y''=(O'',\cD'',\ua'')$ are {\it compatible} if  $O'$ and $O''$ commute. As expected, we say that a finite set of quantum variables is {\it incompatible} if
it contains (at least) an incompatible pair.
Otherwise, the set is said to be {\it compatible}.
Clearly, the empty set is compatible and so is each singleton set.
Moreover, every superset of an incompatible set of quantum variables is also incompatible.

Given compatible pqvs $Y'=(O',\cD',\ua')$ and $Y''=(O'',\cD'',\ua'')$, we consider the pqv
$$Y=(O' \times O'',\cD,\ua)$$
where
$\cD$ is composed by the sets 
$$D_{D'D''}=\{\lambda_{\omega'\omega''}: \omega' \in D' \text{ and } \omega'' \in D''\}$$
for each $D' \in \cD'$ and  $D'' \in \cD''$, and
$\ua$ is $D_{\ua'\ua''}$.

The joint probability distribution of the random variables induced by $Y'$ and $Y''$  and quantum state $\psi$ is such that, for each $\lrt{D',D''} \in \cD' \times \cD''$,
$$\Prob \left(\lrt{Y'_\psi,Y''_\psi} = \lrt{D',D''}\right) =  
    \Prob(Y_\psi=D_{D'D''})=
	\langle \psi, \Proj'_{D'}\Proj''_{D''}\psi\rangle.$$

If a pair of quantum variables is incompatible,
then we cannot, in general,  use them simultaneously  for observing  the quantum system at hand. 
This fact will have an enormous impact on the nature of the envisaged logic. 
Namely, it is at the heart of the need for an epistemic component in the logic.
Actually, this is the key difference between observing a probabilistic system and observing a quantum system.

\section{Language, semantics and calculi}\label{sec:plqo}

The objective of this section is to define an enrichment of classical propositional logic ($\CPL$) to reason about the random results of simultaneously applying observables to a given quantum state. 

The idea is to add as little as possible to the classical propositional language.
To this end, we add a constructor for expressing that (the truth value of) a (classical) formula representing an assertion on the results of simultaneous measurements is knowable and a symbolic construct for constraining the probability of such a formula.

Concerning $\CPL$ we need to adopt some notation and review some concepts and facts.

Let $\Lc$ be the classical language
built with the propositional symbols in $B=\{B_j:j\in\nats\}$ and with $\lverum$ (verum),
using the connectives $\lneg$ (negation) and $\limp$ (implication). 
The other connectives as well as $\lfalsum$ (falsum) are introduced as abbreviations.
Concerning non-primitive connectives, 
we use $\lconj$ (conjunction), $\ldisj$ (disjunction) and $\leqv$ (equivalence) in the paper. 
It becomes handy to use $B_\alpha$ for the set of propositional symbols that occur in a formula $\alpha$.

Recall that a valuation (on $A$) is a map $v : A \to \{0,1\}$ where $A \subseteq B$. Given a valuation $v: A \to \{0,1\}$ and $A' \subseteq A$
we denote by $v|_{A'}$ the restricton of $v$ to $A'$.
We use $v \satc \alpha$ for stating that valuation $v: A \to \{0,1\}$, with $B_\alpha \subseteq A$, satisfies formula $\alpha$ 
and
$\Delta \entc \alpha$ for stating that $\Delta$ classically entails $\alpha$.
One says that $\alpha$ is a tautology if $\entc \alpha$.

Given a valuation $v$ on $A\subseteq B$ and $B_j \in A$, we denote by
$$v^{B_j \mapsfrom d}$$ the valuation such that 
$$v^{B_j \mapsfrom d}(B_i)=\begin{cases}
v(B_i) & \text{ if } i \neq j\\
d & \text{ otherwise.}
\end{cases}
$$

A propositional symbol $B_j$ is said to be {\it essential} for a formula $\alpha$ whenever
$B_j \in B_\alpha$ and there exists a valuation $v$ on $B_\alpha$ such that
 $$v^{B_j \mapsfrom 0} \satc \alpha \quad \text{ iff } \quad v^{B_j \mapsfrom 1} \not \satc \alpha.$$
 We denote by
 $$\eB_\alpha$$
 the set of essential propositional symbols for $\alpha$.
 Observe that $\eB_{\alpha}$ contains only the propositional symbols 
that contribute in an essential way to the truth value of $\alpha$.
Clearly, $\eB_{\alpha} \subseteq B_\alpha$.
Moreover, $\alpha$ is a tautology if and only if $\eB_{\alpha}=\emptyset$. 
That is, the truth values of the propositional symbols  are irrelevant for determining the truth value of $\alpha$.
The set $\eB_\alpha$ can be effectively calculated by finding the algebraic normal form of $\alpha$ and considering its propositional symbols (all of them essential). For more details on algorithms for obtaining the algebraic normal form
see~\cite{zhe:27,ste:pos:07}. So, from now on, we only work with formulas in the algebraic normal form. Hence, 
$$\eB_\alpha=B_\alpha.$$

Before proceeding with the presentation of the envisaged probabilistic logic of quantum observations $\PLQO$, 
we need to adopt some notation and review some facts
concerning also the first-order theory of real closed ordered fields ($\RCOF$),
having in mind the use of its terms for denoting probabilities and other quantities.

Recall that the first-order signature of $\RCOF$ contains 
the constants $\num0$ and $\num1$, 
the unary function symbol $-$, 
the binary function symbols $+$ and $\times$, 
and the binary predicate symbols $=$ and $<$. 
We take the set $$X=X_\nats \cup X_\Lc \cup \{x_{A'}: A' \subseteq B, |A'|=2\},$$
where $X_\nats = \{ x_k: k \in \nats \}$ and $X_\Lc = \{ x_\alpha: \alpha \in \Lc \}$,
as the set of variables. 
By $T^\circ_\RCOF$ we mean the set of terms in $\RCOF$ that do not use variables in $X_\Lc \cup \{x_{A'}: A' \subseteq B, |A'|=2\}$ 
while $T_\RCOF$ denotes the set of all terms.
As we shall see, 
the variables in $X_\Lc$ become handy as fresh variables in the proposed axiomatization $\PLQO$,
for representing within the language of $\RCOF$ the probability of $\alpha$ being true. On the other hand, the variables in $\{x_{A'}: A' \subseteq B, |A'|=2\}$
are used for representing whether or not the observables associated with the propositional symbols in $A'$ are compatible.

As usual, we may write $t_1 \neq t_2$ for $\lneg (t_1 = t_2)$, 
$t_1 \leq t_2$ for $(t_1 < t_2) \ldisj (t_1=t_2)$, $t_1> t_2$ for $\lneg (t_1 \leq t_2)$,
$t_1\geq t_2$ for $\lneg (t_1 < t_2)$,
$t_1\,t_2$ for $t_1 \times t_2$ and $t^n$ for 
$$\underbrace{t \times \dots \times t}_{n \text{ times}}.$$
Furthermore, we also use the following abbreviations for any given $m\in\pnats$ and $n\in\nats$:
\begin{itemize}

\item $\num m$ for $\underbrace{\num1+\dots +\num1}_{\text{addition of } m \text{ units}}$;

\item $\inv{\num m}$ for the unique $z$ such that $\num m \times z = \num1$;

\item $\ds\frac{\num n}{\num m}$ for $\inv{\num m} \times \num n$.

\end{itemize}
The last two abbreviations might be extended to other terms, but we need them only for numerals. 
We do not notationally distinguish, for the sake of a lighter style, between a natural number and the corresponding numeral and between each predicate symbol and its denotation.

In order to avoid confusion with the other notions of satisfaction used herein, we adopt $\satfo$ for denoting satisfaction in first-order logic (over the language of $\RCOF$).

Recall also that the theory $\RCOF$ is decidable~\cite{tar:51}. This fact will be put to good use in the axiomatization for $\PLQO$ (presented at the end of this section).
Furthermore, every model of $\RCOF$ satisfies a formula in the $\RCOF$ language if and only the formula is a theorem of $\RCOF$ (Corollary~3.3.16 in~\cite{mar:02}).
We shall take advantage of this result in the semantics of $\PLQO$ for adopting the ordered field $\reals$ of the real numbers as the model of $\RCOF$.

\subsection*{Language}

The language $L_\PLQO$ of the probabilistic logic of quantum observations $\PLQO$ is inductively defined as follows:

\begin{itemize}
\item $\lO \alpha \in L_\PLQO$ 
		when $\alpha \in \Lc$;
\item $\inte{\alpha} \mycirc p \in L_\PLQO$ 
		when $\alpha \in \Lc$,  $p \in T^\circ_\RCOF$ and $\mycirc \in \{=,<,\neq,\leq, >, \geq\}$;
\item $\varphi_1 \limp \varphi_2 \in L_\PLQO$ 
		when $\varphi_1,\varphi_2 \in L_\PLQO$.
\end{itemize}

Propositional abbreviations can be introduced as usual. For instance,
$$\lneg \varphi \;\text{ for }\; \varphi \limp (\inte{\lverum} < 1)$$
and similarly for $\lconj$, $\ldisj$ and $\leqv$.

\subsection*{Semantics}

Given a term $t$ in $T_\RCOF$ and an assignment $\rho: X \to \reals$, 
we write $t^{\reals\rho}$ for the denotation of term $t$ in $\reals$ for $\rho$. 
When $t$ does not contain variables we may use $t^\reals$ for the denotation of $t$ in $\reals$.

By a {\it quantum interpretation structure} (in short, {\it quantum structure}) we mean a triple
$$I=(\cH,\psi,B_j \mapsto \udl B_j)$$ 
where:
\begin{itemize}
\item $\cH$ is a complex separable Hilbert space;
\item $\psi$ is a unit vector in $\cH$;
\item $B_j \mapsto \udl B_j$ is a map that associates to each propositional symbol $B_j$
	a propositional quantum variable $\udl B_j=(O_j,\cD_j,\ua_j)$ defined on $\cH$ where $|\cD_j|=2$.
\end{itemize}

In the sequel, given $A \subseteq B$ and a valuation $v: A \to \{0,1\}$, let 
$$D^v_{{j}} = 
\begin{cases}
	\ua_{j} &\textrm{if } v(B_{j})=1\\
	\da_{j} &\textrm{otherwise}
\end{cases}$$
for each $B_j \in A$.
Therefore, we are taking $\ua$ as true and $\da$ as false. 

 Given $\alpha \in \Lc$, let $\udl\alpha = \{ \udl B_j: B_j \in B_{\alpha} \}.$
We say that $\alpha$ is an {\it $I$-observable formula} 
if $\udl\alpha$ is a compatible set of propositional quantum variables. Observe that only the essential propositional symbols are relevant.

A quantum structure $I$ assigns a probability to any formula $\alpha \in \Lc$ such that $\udl{\alpha}$ is a compatible set 
as follows:
 \begin{align*}  
\Prob_I(\alpha) = 
		\sum_{\text{\shortstack[c]
{$v: B_{\alpha} \to \{0,1\}$\\
$v\satc\alpha$
}
}} 
		\Prob \left( \lrt{ (\udl B_{j_1})_{\psi},\dots,(\udl B_{j_k})_{\psi}} =  
					\lrt{ D^v_{{j_1}},\dots,D^v_{{j_k}} } \right)
\displaybreak[1] \\[1mm]
& \hspace*{-102mm} \displaystyle 
= \sum_{\text{\shortstack[c]
{$v: B_{\alpha} \to \{0,1\}$\\
$v\satc\alpha$
}
}}   \langle \psi,\Proj^{j_1}_{D^v_{{j_1}}} \dots  \Proj^{j_k}_{D^v_{{j_k}}}\psi\rangle.
\displaybreak[1]
\end{align*}
assuming that $B_{\alpha} = \{B_{j_1},\dots,B_{j_k}\}$. We prove in Section~\ref{sec:ssound} that $\Prob_I$ is indeed a {\it probability assignment} (see~\cite{adam:98,acs:jfr:css:15}) when the formula at hand involves only  a compatible set of observables according to $I$, that is,  a map $P$ on $\Lc$ satisfying the following (Adams) principles:
\begin{itemize}
\item [P1] $0 \leq P(\alpha) \leq 1$;
\item [P2] If $\entc \alpha$ then $P(\alpha)=1$;
\item [P3] If $\alpha \entc \beta$ then $P(\alpha) \leq P(\beta)$;
\item [P4] If $\entc \lneg (\beta \lconj \alpha)$ then $P(\beta \ldisj \alpha)= P(\beta) + P(\alpha)$.
\end{itemize} 
In~\cite{acs:jfr:css:15} we showed that P3 follows from P1, P2 and P4. 

As we proceed to explain, a quantum structure $I$ tells us, among other things, which are the knowable classical formulas 
in the sense that they do not involve incompatible sets of quantum variables.

Given a quantum structure $I$ and an assignment $\rho: X \to \reals$,
{\it satisfaction of  formulas} by $I$ and $\rho$  is inductively defined as follows:
\begin{itemize}

\item $I\rho \sat \lO\alpha$ whenever $\alpha$ is $I$-observable;

\item $I\rho \sat \inte{\alpha} \mycirc p$ whenever 
	$I\rho \sat \lO\alpha$
	and
	$\Prob_I(\alpha) \mycirc p^{\reals\rho}$;
	
\item $I\rho \sat \varphi_1 \limp \varphi_2$ whenever $I\rho \not \sat \varphi_1$ or $I\rho \sat  \varphi_2$.
\end{itemize}
We may omit the reference to the assignment $\rho$ whenever the $\PLQO$ formula in hand does not include variables.
For examples illustrating satisfaction, see Section~\ref{sec:epistemic}.

Let $\Gamma \subseteq L_\PLQO$ and $\varphi \in L_\PLQO$. We say that $\Gamma$ {\it entails} $\varphi$, written $\Gamma \ent \varphi$,
whenever, for every quantum structure $I$ and assignment $\rho$,
if $I\rho \sat \gamma$ for each $\gamma \in \Gamma$ then $I\rho \sat \varphi$.
As expected, $\varphi$ is said to be {\it valid} when $\ent \varphi$.



Observe that entailment in $\PLQO$ is not compact. Indeed, since $\reals$ is Archimedean and $\lO B_1$ is valid,
$$\left\{\inte{B_1} \leq \frac 1n : n \in \nats\right\} \ent \inte{B_1}=0.$$
However, there is no finite subset $\Psi$ of $\{\inte{B_1} \leq \frac 1n : n \in \nats\}$ such that
$$\Psi \ent \inte{B_1}=0.$$

\subsection*{Calculus}

The $\PLQO$ calculus combines propositional reasoning with $\RCOF$ reasoning. We intend to use the $\RCOF$ reasoning to a minimum, namely to prove assertions like
$$\left(\beta_1 \lconj \cdots \lconj   \beta_k\right) \limp
									\beta$$
where $\beta_1,\dots,\beta_k, \beta$ are $\PLQO$ literals. To this end, we need to introduce some preliminary material.
Given $U \subseteq A \subseteq B$ where $A$ is finite, we use the abbreviation
$$\phi_A^U \quad\text{for}\quad
		\left(\bigwedge_{B_j \in U} B_j\right) \lconj \left(\bigwedge_{B_j \in A \setminus U} \lneg B_j\right),$$
		 when $A \neq \emptyset$,
and,  for $\lverum$, otherwise. Note that each classical formula $\phi^U_A$ identifies the valuation on $A$ that assigns true to the propositional symbols in $U$ and assigns false to the propositional symbols in $A \setminus U$. 

Given a propositional formula $\alpha$, we denote by 
$$Q^{\lO \alpha}$$
the $\RCOF$ formula
 $$\bigwedge_{\text{\shortstack[c]
{$A' \subseteq B_\alpha$\\
$|A'|=2$
}
}}x_{A'}=0,$$
and by
$$Q^{{\inte{}}{\alpha}}$$
the $\RCOF$ formula
\begin{align*} \left(\bigwedge_{U\subseteq {B_\alpha}}
0 \leq x_{\phi_{{B_\alpha}}^{U} } \leq 1\right)\lconj
	\left(\sum_{U\subseteq {B_\alpha}} x_{\phi_{{B_\alpha}}^{U} }=1\right)
	\lconj  \left(\bigwedge_{\text{\shortstack[c]
{$A' \subseteq {B_\alpha}$\\
$U' \subseteq A'$
}
}} x_{\phi_{A'}^{U'} } = \sum_{\text{\shortstack[c]
{$U \subseteq {B_\alpha}$\\
$U \cap A' = U'$
}
}} x_{\phi_{{B_\alpha}}^{U}}\right) \lconj 
\displaybreak[1] \\[1mm]
& \hspace*{-97mm} \displaystyle  
\left( x_{\alpha} = 
\sum_{\text{\shortstack[c]
{$v:B_{\alpha} \to \{0,1\}$\\
$v\satc\alpha$
}
}}  x_{\phi^{\{B_j \in B_{{\alpha}}: v(B_j)=1\}}_{B_{{\alpha}}}}\right).
\end{align*}
We proceed to explain the meaning of each conjunct of formula $Q^{{\inte{}}{\alpha}}$. The idea is that
each variable $x_\gamma$ represents the probability of $\gamma$ being true in a particular quantum structure which, according to the definition,
is the sum of the probabilities of obtaining the results induced by the valuations that satisfy $\gamma$ when performing the measurements
associated with the formula (assuming that the underlying observables are compatible). When $\gamma$ is
$\phi^U_{B_\alpha}$ there is a unique valuation on ${B_\alpha}$ that satisfies $\gamma$ and so we can look  at $x_\gamma$ as the probability 
of obtaining the results induced by that valuation. 
The conjunct
$$\bigwedge_{U\subseteq {B_\alpha}}
0 \leq x_{\phi_{{B_\alpha}}^{U} } \leq 1$$
imposes that the probability of obtaining the results induced by each possible valuation on $A$ is in the interval $[0,1]$. The conjunct
$$\sum_{U\subseteq {B_\alpha}} x_{\phi_{{B_\alpha}}^{U} }=1$$
expresses that the sum of those probabilities is one. The conjunct
$$x_{\phi_{A'}^{U'} } = \sum_{\text{\shortstack[c]
{$U \subseteq {B_\alpha}$\\
$U \cap A' = U'$
}
}} x_{\phi_{{B_\alpha}}^{U}}$$
is a marginal probability condition, and the conjunct
$$x_{\alpha} = 
\sum_{\text{\shortstack[c]
{$v:B_{\alpha} \to \{0,1\}$\\
$v\satc\alpha$
}
}}  x_{\phi^{\{B_j \in B_{{\alpha}}: v(B_j)=1\}}_{B_{{\alpha}}}}$$
expresses that the probability of $\alpha_i$ being true in a quantum structure is the sum of the probabilities of obtaining the results induced
by each valuation that satisfies the formula. 

With respect to $Q^{\lO \alpha}$, each variable $x_{A'}$ is intended to represent whether or not the observables associated with the propositional symbols in $A'$ are compatible.

We are now ready to define an equivalent  $\RCOF$ formula for each $\PLQO$ formula. 
Given a formula $\beta$ of $\PLQO$ of the form $\inte{\alpha} \mycirc p$ or $\lO \alpha$ or $\lneg \lO \alpha$, let $\beta^{\RCOF}$ be the $\RCOF$ formula 
$$
\begin{cases}
Q^{\lO \alpha} \lconj {Q^{{\inte{}}{\alpha}}} \lconj (x_{\alpha} \mycirc p) & \text{if } \beta \text{ is } \inte{\alpha} \mycirc p\\[1mm]
 Q^{\lO \alpha}  \lconj {Q^{{\inte{}}{\alpha}}}& \text{if } \beta \text{ is } \lO \alpha\\[1mm]
 \lneg Q^{\lO \alpha} & \text{if } \beta \text{ is } \lneg \lO \alpha.
\end{cases}
$$

Moreover, we define $\varphi^{\RCOF}$, where $\varphi$ is a formula of $\PLQO$, inductively as follows:
\begin{itemize}
\item $\varphi^{\RCOF}$ is $\beta^{\RCOF}$ whenever $\varphi$ is either $\inte{\alpha} \mycirc p$ or $\lO \alpha$ or $\lneg \lO \alpha$;
\item $\varphi^{\RCOF}$ is $\varphi_1^{\RCOF} \limp \varphi_2^{\RCOF}$ whenever $\varphi$ is $\varphi_1 \limp \varphi_2$ and is not of the form $\lneg \lO \alpha$.
\end{itemize}

The calculus for $\PLQO$ embodies classical reasoning and makes use of $\RCOF$ theorems.
It contains the following axioms and rules, assuming that $\varphi_1,\varphi_2$ are formulas and  
$\beta_1, \cdots, \beta_k, \beta$ are formulas of the form 
$\lO \alpha$, or $\lneg \lO \alpha$ or  $\inte{\alpha} \mycirc p$:

\begin{itemize}

\item[$\TT$] $\ds\lrule{}{\varphi_1}$\\[2mm]
provided that
$\varphi_1$ is a tautological formula;\\[-3mm]

\item[$\MP$] $\ds\lrule{\begin{array}{l}
							\varphi_1 \\
							\varphi_1 \limp\varphi_2
					\end{array}}
					{\varphi_2}$;\\[-1mm]
 
\item[$\RR$] $\ds\lrule{}
{(\beta_1 \lconj \cdots \lconj  \beta_k) \limp \beta}$\;
provided that \\[2mm]
$$\ds\bigforall
		\left(\left(
\bigwedge_{j=1}^k \beta_j^{\RCOF}\right)
					\limp \beta^{\RCOF}\right)$$
 is a theorem of $\RCOF$.	

\end{itemize}

Axioms $\TT$ and rule $\MP$ extend classical reasoning to $\PLQO$ formulas. 
 Rule $\RR$ import all we need from $\RCOF$ for deriving $\PLQO$ assertions on quantum observability.
 
\subsection*{Examples}
				
We start by showing that
$$\der{(\lO\alpha_1) \leqv (\lO\alpha_2)}$$
provided that $\entc \alpha_1 \leqv \alpha_2$, that is, observability is preserved by classical equivalence. 
Observe that
$\eB_{\alpha_1}=\eB_{\alpha_2}$
since $\entc \alpha_1 \leqv \alpha_2$. Thus,  $B_{\alpha_1}=B_{\alpha_2}$.
Hence,
$$Q^{\lO \alpha_1} \text{ is } Q^{\lO \alpha_2} \text{ and } {Q^{{\inte{}}{\alpha_1}}} \text{ is } {Q^{{\inte{}}{\alpha_2}}}$$
and so, $$\ds\bigforall
		\left(
		(Q^{\lO \alpha_1} \lconj {Q^{{\inte{}}{\alpha_1}}})
					\limp (Q^{\lO \alpha_2}\lconj {Q^{{\inte{}}{\alpha_2}}})\right)$$
					and
$$\ds\bigforall
		\left(
		(Q^{\lO \alpha_2} \lconj {Q^{{\inte{}}{\alpha_2}}})
					\limp (Q^{\lO \alpha_1}\lconj {Q^{{\inte{}}{\alpha_1}}})\right)$$	
 are theorems of $\RCOF$. Thus, we have the derivation in Figure~\ref{fig:der1}.
\begin{figure}[t]
$$\begin{array}{rlr}
	1 & \ds\bigforall
		\left(
		(Q^{\lO \alpha_1} \lconj {Q^{{\inte{}}{\alpha_1}}})
					\limp (Q^{\lO \alpha_2}\lconj {Q^{{\inte{}}{\alpha_2}}})\right) &\RCOF\\[3mm]			              
	2 & (\lO\alpha_1) \limp (\lO\alpha_2)	& \RR \; 1\\[3mm]
	3 &  \ds\bigforall
		\left(
		(Q^{\lO \alpha_2} \lconj {Q^{{\inte{}}{\alpha_2}}})
					\limp (Q^{\lO \alpha_1}\lconj {Q^{{\inte{}}{\alpha_1}}})\right) &  \RCOF\\[3mm]
	4 &  (\lO\alpha_2) \limp (\lO\alpha_1)	& \RR \; 3\\[3mm]
	5 &  ((\lO\alpha_1) \limp (\lO\alpha_2)) \limp (( (\lO\alpha_2) \limp (\lO\alpha_1)	) \limp ((\lO\alpha_1) \leqv (\lO\alpha_2)))& \TT	\\[3mm]
	6 & 	( (\lO\alpha_2) \limp (\lO\alpha_1)) \limp ((\lO\alpha_1) \leqv (\lO\alpha_2)) & \MP \; 2,5\\[3mm]
	7 & (\lO\alpha_1) \leqv (\lO\alpha_2) & \MP\;  4,6
\end{array}$$\vspace*{-6mm}
\caption{$\dummy\der (\lO\alpha_1) \leqv (\lO\alpha_2)$ provided that $\entc \alpha_1 \leqv \alpha_2$.}\label{fig:der1}
\end{figure}

As another example, we now prove that 
$$\der \lO \lverum,$$
that is, that $\lverum$ is always observable.
Indeed,
$$\ds\bigforall
		\left( 
		\lverum
					\limp (\lO \lverum)^{\RCOF}\right)$$
 is a theorem of $\RCOF$ since
 $(\lO \lverum)^{\RCOF}$ is $Q^{\lO\lverum} \lconj {Q^{{\inte{}}{\lverum}}}$ which is $\lverum$.
Therefore, by $\RR$, 
					we have that $\lO \lverum$ is a theorem.
					
As a consequence of these two examples, every tautology is observable taking into account the two given examples. 

Finally, we prove that
$$\lO (B_1 \lconj B_2)  \der\inte{(B_1 \lconj B_2)} \geq 0.$$
We start by verifying that 
$$\ds\bigforall
		\left(
		(\lO (B_1 \lconj B_2))^\RCOF
					\limp (\inte{(B_1 \lconj B_2)}\geq 0)^\RCOF\right)$$
is a theorem of $\RCOF$. 
Indeed,
\begin{itemize}
\item $(\lO (B_1 \lconj B_2))^\RCOF$ is $Q^{\lO (B_1 \lconj B_2)} \lconj {Q^{{\inte{}}{B_1 \lconj B_2}}}$;
\item $(\inte{(B_1 \lconj B_2)}\geq 0)^\RCOF$ is 
$Q^{\lO (B_1 \lconj B_2)} \lconj Q^{{\inte{}}{B_1 \lconj B_2}} \lconj (x_{B_1 \lconj B_2} \geq 0)$.

\end{itemize}
Thus, we have the derivation in Figure~\ref{fig:der2}.
\begin{figure}[t]
$$\begin{array}{rlr}
	1 & \lO (B_1 \lconj B_2) &\HYP\\[3mm]			              
	2 & \ds\bigforall
		\left(
		\lO (B_1 \lconj B_2))^\RCOF
					\limp \right.
					\left. (\inte{(B_1 \lconj B_2)}\geq 0)^\RCOF\right)	& \RCOF\\[3mm]
	3 & (\lO (B_1 \lconj B_2)) \limp (\inte{(B_1 \lconj B_2)} \geq 0)&  \RR \; 2\\[3mm]
	4 &  \inte{(B_1 \lconj B_2)} \geq 0& \MP \; 1,3
\end{array}$$\vspace*{-6mm}
\caption{$\dummy\lO (B_1 \lconj B_2)  \der\inte{(B_1 \lconj B_2)} \geq 0$.}\label{fig:der2}
\end{figure}
%

\section{Strong soundness}\label{sec:ssound}

In this section we show that the calculus for $\PLQO$ is strongly sound, starting by proving some auxiliary results namely that
$\Prob_I$ is a probability assignment. 

Let  $I$ be a quantum structure, $A \subseteq B$ a set such that
$\{\udl{B}_j: B_j \in A\}$ is a compatible set of propositional quantum variables and $v$ a valuation on $A$. 
Given an  orthonormal basis 
$U$ of $\cH$ composed of common eigenvectors of the observables in $\{\udl{B}_j: B_j \in A\}$ (observe that we assume that observables are diagonizable and bounded), 
we denote by $$U_v$$ the set
$\{{u}\in U: O_{j}{u}=\lambda_{j}{u}, \lambda_{j} \in D^v_{B_{j}}, B_j \in A\}.$

\begin{prop}\em \label{prop:partition}
Given a quantum structure $I$,  sets $A \subseteq A'  \subseteq B$ such that $\{\udl{B}_j: B_j \in A'\}$ is a compatible set of propositional quantum variables,  an orthonormal basis $U$ of $\cH$ composed of common eigenvectors of the observables in $\{\udl{B}_j: B_j \in A'\}$,  and $v: A \to \{0,1\}$, 
then 
$$\{U_{v'}:v': A' \to \{0,1\} \text{ and } v'|_{A} = v\}$$
is a partition of $U_v$.
\end{prop}
\begin{proof}\ \\
(1)~$U_{v'_1} \cap U_{v'_2}= \emptyset$ for every $v'_1,v'_2: A' \to \{0,1\}$ such that $v'_1|_{A} = v$ and $v'_2|_{A} = v$ and $v'_1 \neq v'_2$. \\[1mm]
Indeed, suppose, by contradiction that $U_{v'_1} \cap U_{v'_2}\neq \emptyset$.
Let ${u'} \in U_{v'_1} \cap U_{v'_2}$. 
Since $v'_1 \neq v'_2$ then there is $B_{j'} \in A'$ such that $v'_1(B_{j'}) \neq v'_2(B_{j'})$.
Hence $D^{v'_1}_{B_{j'}} \neq D^{v'_2}_{B_{j'}}$ and so $D^{v'_1}_{B_{j'}} \cap D^{v'_2}_{B_{j'}}=\emptyset$.
Thus, there is no $\lambda_{j'}$ such that $O_{j'}{u'}=\lambda_{j'}{u'}$ and $\lambda_{j '}\in D^{v'_1}_{B_{j'}}$ and $\lambda_{j'} \in D^{v'_2}_{B_{j'}}$
contradicting the assumption that ${u'} \in U_{v'_1}$ and ${u'} \in U_{v'_2}$.\\[2mm]
(2)~
$U_v \subseteq \bigcup_{v': A' \to \{0,1\}, v'|_{A} = v} U_{v'}$. \\[1mm]
Indeed, 
let ${u} \in U_v$. Then, for every $B_j \in A$ we have $O_j {u} = \lambda_{j}{u}$ and $\lambda_{j} \in D^v_{B_{j}}$.
Let $v': A' \to \{0,1\}$ be the valuation such that
$$v'(B_{j'})= \begin{cases}
1 & \text{ if }  \lambda_{j'} \in \ua_{j'}\\
0 & \text{ otherwise}
\end{cases}
$$
for each $B_{j'} \in A'$ and $\lambda_{j'}$ such that 
 $O_{j'} {u} = \lambda_{j'} {u}$. Then, it is immediate to see that ${u} \in U_{v'}$ by definition of $v'$. It remains to show that
$v'|_{A} = v$. Indeed, let $B_j \in A$. There are two cases.\\
(a) $v'|_{A}(B_j)=1$. Then, $O_{j} {u} = \lambda_{j} {u}$ and $\lambda_{j} \in \ua_{j}$. Hence,
$D^v_{B_{j}}= \ua_{j}$ since ${u} \in U_v$. So, $v(B_j)=1$.\\[1mm]
(b) $v'|_{A}(B_j)=0$. We omit the proof since it is similar to (1).\\[2mm]
(3)~$U_{v'} \subseteq U_v$ for every $v': A' \to \{0,1\}$ such that $v'|_{A} = v$. Let $v': A' \to \{0,1\}$ be an arbitrary valuation such that $v'|_{A} = v$. Let ${u'} \in U_{v'}$.
Then, for every $B_{j'} \in A'$ we have $O_{j'} {u'} = \lambda_{j'}{u'}$ and $\lambda_{j'} \in D^{v'}_{B_{j'}}$.
Let $B_j \in A$. Hence, $O_{j} u' = \lambda_{j}{u'}$ and $\lambda_{j} \in D^{v'}_{B_{j}}$.
Observe that $v'(B_j)=v(B_j)$. So $D^{v'}_{B_{j}}=D^{v}_{B_{j}}$. Therefore,
${u'} \in U_v$.
\end{proof}

Note that, given a quantum structure $I=(\cH,\psi, B_j \mapsto \udl{B}_j)$,  a set $\{B_{j_1},\dots,B_{j_k}\}$ such that $\{\udl{B}_{j_i}: 1\leq i \leq k\}$ is a compatible set,  an orthonormal basis $U$ of $\cH$ composed of common eigenvectors of the observables in $\{\udl{B}_{j_i}: 1\leq i \leq k\}$ and a valuation $v: \{B_{j_1},\dots,B_{j_k}\} \to \{0,1\}$,  
then $$\Proj^{j_k}_{D^v_{{j_k}}}\ldots \Proj^{j_1}_{D^v_{{j_1}}}\psi\quad \text{is} \quad 
\ds \sum_{{u} \in U_v} \langle \psi{,}u \rangle{u}.$$

\begin{prop}\em \label{prop:probIatr}
Given a quantum structure  $I=(\cH,\psi, B_j \mapsto \udl{B}_j)$, $\alpha \in \Lc$ such that $\udl{\alpha}$ is a compatible set,  an orthonormal basis $U$ of $\cH$ composed of common eigenvectors of the observables in $\udl{\alpha}$, then
$$\Prob_I(\alpha)= \sum_{\text{\shortstack[c]
{$v: B_{\alpha} \to \{0,1\}$\\
$v\satc\alpha$
}
}} \left(\sum_{{u} \in U_v} \langle \psi{,}u \rangle^2\right).$$
\end{prop}
\begin{proof}\ \\  
Assume without loss of generality that
$B_\alpha =\{B_1,\dots,B_k\}$.
Then, 
\begin{align*}  
 \Prob_I(\alpha) = 
		\sum_{\text{\shortstack[c]
{$v: B_{\alpha} \to \{0,1\}$\\
$v\satc\alpha$
}
}} 
		\Prob \left( \lrt{ (\udl B_{1})_{\psi},\dots,(\udl B_{k})_{\psi}} =  
					\lrt{ D^v_{{1}},\dots,D^v_{{k}} } \right)
					\displaybreak[1] \\[1mm]
& \hspace*{-80mm} \displaystyle 
= \sum_{\text{\shortstack[c]
{$v: B_{\alpha} \to \{0,1\}$\\
$v\satc\alpha$
}
}} \langle \psi, \Proj_{D^v_{k}} \ldots \Proj_{D^v_{1}}\psi\rangle 
\displaybreak[1] \\[1mm]
& \hspace*{-80mm} \displaystyle 
= \sum_{\text{\shortstack[c]
{$v: B_{\alpha} \to \{0,1\}$\\
$v\satc\alpha$
}
}} \langle \psi, \sum_{{u} \in U_v} \langle \psi{,}u \rangle {u}\rangle
\displaybreak[1] \\[1mm]
& \hspace*{-80mm} \displaystyle 
= \sum_{\text{\shortstack[c]
{$v: B_{\alpha} \to \{0,1\}$\\
$v\satc\alpha$
}
}} \left(\sum_{{u} \in U_v} \langle \psi{,}u \rangle^2\right).
\displaybreak[1] 
\end{align*}
\end{proof}

\begin{prop}\em \label{prop:entimpliesleqcomp} Given a quantum structure  $I=(\cH,\psi, B_j \mapsto \udl{B}_j)$ and   
$\alpha,\gamma \in \Lc$ such that $B_{\gamma} \subseteq B_{\alpha}$ and $\udl{\alpha}$ is a compatible set, then $\alpha \entc \gamma$ implies $\Prob_I(\alpha)\leq \Prob_I(\gamma)$. Moreover, 
$\gamma \entc \alpha$ implies $\Prob_I(\gamma)\leq \Prob_I(\alpha)$.
\end{prop}
\begin{proof}\ \\ 
Let 
$U$ be an orthonormal basis of $\cH$ composed of common eigenvectors of the observables in $\udl{\alpha}$.\\[1mm]
We start by showing that $\alpha \entc \gamma$ implies $\Prob_I(\alpha)\leq \Prob_I(\gamma)$. 
Assume that $\alpha \entc \gamma$. Then, 
\begin{align*}  
 \Prob_I(\alpha) =  \sum_{\text{\shortstack[c]
{$v: B_{\alpha} \to \{0,1\}$\\
$v\satc\alpha$
}
}} \left(\sum_{{u} \in U_v} \langle \psi{,}u \rangle^2\right)
\displaybreak[1] \\[1mm]
& \hspace*{-55mm} \displaystyle
\leq \sum_{\text{\shortstack[c]
{$v: B_{\alpha} \to \{0,1\}$\\
$v\satc\gamma$
}
}} \left(\sum_{{u} \in U_v} \langle \psi{,}u \rangle^2\right)
\displaybreak[1] \\[1mm]
& \hspace*{-55mm} \displaystyle 
= \sum_{\text{\shortstack[c]
{$v: B_{\gamma} \to \{0,1\}$\\
$v\satc\gamma$
}
}} \left(
\sum_{\text{\shortstack[c]
{$v':B_{\alpha} \to \{0,1\}$\\
$v'|_{B_{\gamma}} = v$
}
}}  \sum_{\ket{u'} \in U_{v'}} \langle \psi{,}u' \rangle^2\right) 
\displaybreak[1] \\[1mm]
& \hspace*{-55mm} \displaystyle 
=  \sum_{\text{\shortstack[c]
{$v:B_{\gamma}\to \{0,1\}$\\
$v\satc\gamma$
}
}} \left(\sum_{{u} \in U_{v}} \langle \psi{,}u \rangle^2\right) \quad (*)
\displaybreak[1] \\[1mm]
& \hspace*{-55mm} \displaystyle 
=  \Prob_I(\gamma)
\end{align*}
where $(*)$ follows from Proposition~\ref{prop:partition}.\\[1mm]
We now show that
$\gamma \entc \alpha$ implies $\Prob_I(\gamma)\leq \Prob_I(\alpha)$.
Assume that $\gamma \entc \alpha$.
Then, 
\begin{align*}  
 \Prob_I(\gamma) =  \sum_{\text{\shortstack[c]
{$v: B_{\gamma} \to \{0,1\}$\\
$v\satc\gamma$
}
}} \left(\sum_{{u} \in U_v} \langle \psi{,}u \rangle^2\right)
\displaybreak[1] \\[1mm]
& \hspace*{-55mm} \displaystyle 
= \sum_{\text{\shortstack[c]
{$v: B_{\gamma} \to \{0,1\}$\\
$v\satc\gamma$
}
}} \left(
\sum_{\text{\shortstack[c]
{$v':B_{\alpha} \to \{0,1\}$\\
$v'|_{B_{\gamma}} = v$
}
}}  \sum_{\ket{u'} \in U_{v'}} \langle \psi{,}u' \rangle^2\right) \quad (*)
\displaybreak[1] \\[1mm]
& \hspace*{-55mm} \displaystyle
= \sum_{\text{\shortstack[c]
{$v: B_{\alpha} \to \{0,1\}$\\
$v\satc\gamma$
}
}} \left(\sum_{{u} \in U_v} \langle \psi{,}u \rangle^2\right)
\displaybreak[1] \\[1mm]
& \hspace*{-55mm} \displaystyle 
\leq  \sum_{\text{\shortstack[c]
{$v:B_{\alpha}\to \{0,1\}$\\
$v\satc\alpha$
}
}} \left(\sum_{{u} \in U_{v}} \langle \psi{,}u \rangle^2\right)
\displaybreak[1] \\[1mm]
& \hspace*{-55mm} \displaystyle 
=  \Prob_I(\alpha)
\end{align*}
where $(*)$ follows from Proposition~\ref{prop:partition}.
\end{proof}

\begin{prop}\em \label{prop:probass} Given a quantum structure  $I=(\cH,\psi, B_j \mapsto \udl{B}_j)$,  
$\Prob_I$ is a probability assignment for every formula $\alpha \in \Lc$ such that $\udl{\alpha}$ is a compatible set.
\end{prop}
\begin{proof}\ \\ 
(1)~$0 \leq \Prob_I(\alpha)\leq 1$, for every formula $\alpha$ such that $\udl{\alpha}$ is a compatible set. Let 
$U$ be an orthonormal basis of $\cH$ composed of common eigenvectors of the observables in $\udl{\alpha}$. Then, 
\begin{align*}  
 \Prob_I(\alpha) =  \sum_{\text{\shortstack[c]
{$v: B_{\alpha} \to \{0,1\}$\\
$v\satc\alpha$
}
}} \left(\sum_{{u} \in U_v} \langle \psi{,}u \rangle^2\right)
\displaybreak[1] \\[1mm]
& \hspace*{-55mm} \displaystyle 
\geq 0
\displaybreak[1]
\end{align*}
and
\begin{align*}  
 \Prob_I(\alpha) =  \sum_{\text{\shortstack[c]
{$v: B_{\alpha} \to \{0,1\}$\\
$v\satc\alpha$
}
}} \left(\sum_{{u} \in U_v} \langle \psi{,}u \rangle^2\right)
\displaybreak[1] \\[1mm]
& \hspace*{-55mm} \displaystyle 
\leq \sum_{v: B_{\alpha} \to \{0,1\}} \left(\sum_{{u} \in U_v} \langle \psi{,}u \rangle^2\right)
\displaybreak[1] \\[1mm]
& \hspace*{-55mm} \displaystyle 
= \sum_{{u} \in U} \langle \psi{,}u \rangle^2 \qquad (*)
\displaybreak[1] \\[1mm]
& \hspace*{-55mm} \displaystyle 
= 1,
\displaybreak[1]
\end{align*}
where $(*)$ holds, by Proposition~\ref{prop:partition} since $\{U_v: v: B_\alpha \to \{0,1\}$ is a partition of $U$.\\[2mm]
(2)~If $\entc \alpha$ then $\Prob_I(\alpha)=1$, for every formula $\alpha$ such that $\udl{\alpha}$ is a compatible set. Let 
$U$ be an orthonormal basis of $\cH$ composed of common eigenvectors of the observables in $\udl{\alpha}$. Assuming that $\alpha$ is a tautology, then
\begin{align*}  
 \Prob_I(\alpha) =  \sum_{\text{\shortstack[c]
{$v: B_{\alpha} \to \{0,1\}$\\
$v\satc\alpha$
}
}} \left(\sum_{{u} \in U_v} \langle \psi{,}u \rangle^2\right)
\displaybreak[1] \\[1mm]
& \hspace*{-55mm} \displaystyle 
= \sum_{v: B_{\alpha} \to \{0,1\}} \left(\sum_{{u} \in U_v} \langle \psi{,}u \rangle^2\right)
\displaybreak[1] \\[1mm]
& \hspace*{-55mm} \displaystyle 
= \sum_{{u} \in U} \langle \psi{,}u \rangle^2  \qquad (*)
\displaybreak[1] \\[1mm]
& \hspace*{-55mm} \displaystyle 
= 1, 
\displaybreak[1]
\end{align*}
where $(*)$ holds, by Proposition~\ref{prop:partition} since $\{U_v: v: B_\alpha \to \{0,1\}$ is a partition of $U$.\\[2mm]
(3)~If $\alpha \entc \beta$ then $\Prob_I(\alpha)\leq \Prob_I(\beta)$, for every formulas $\alpha$ and $\beta$ such that $\udl{\alpha}$ and $\udl{\beta}$ are compatible sets. Assume that $\alpha \entc \beta$. Then, since propositional logic has the Craig interpolation property there is $\gamma$ such $B_\gamma \subseteq B_\alpha \cap B_\beta$ and $\alpha \entc \gamma$ and $\gamma \entc \beta$. Then, by 
Proposition~\ref{prop:entimpliesleqcomp}, $\Prob_I(\alpha) \leq \Prob_I(\gamma)$ and $\Prob_I(\gamma) \leq \Prob_I(\beta)$.\\[2mm]
(4)~If $\entc \lneg (\beta \lconj \alpha)$ then $\Prob_I(\beta \ldisj \alpha)=\Prob_I(\beta)+ \Prob_I(\alpha)$, for every formulas $\alpha$ and $\beta$ such that $\udl{\alpha}\cup \udl{\beta}$ is a compatible set. Assume that $\entc \lneg (\beta \lconj \alpha)$.
Then,
\begin{align*}  
 \Prob_I(\beta \ldisj \alpha) =  \sum_{\text{\shortstack[c]
{$v:B_{\beta \ldisj \alpha} \to \{0,1\}$\\
$v\satc\beta \ldisj \alpha$
}
}} \left(\sum_{{u} \in U_v} \langle \psi{,}u \rangle^2\right)
\displaybreak[1] \\[1mm]
& \hspace*{-55mm} \displaystyle 
= 
\sum_{\text{\shortstack[c]
{$v':B_{\alpha} \cup B_{\beta} \to \{0,1\}$\\
$v' \satc \beta \ldisj \alpha$
}
}}  \left(\sum_{\ket{u'} \in U_{v'}} \langle \psi{,}u' \rangle^2\right)
\displaybreak[1] \\[1mm]
& \hspace*{-55mm} \displaystyle 
= 
\sum_{\text{\shortstack[c]
{$v':B_{\alpha} \cup B_{\beta} \to \{0,1\}$\\
$v' \satc \beta$
}
}}  \left( \sum_{\ket{u'} \in U_{v'}} \langle \psi{,}u' \rangle^2 \right)\; + 
\displaybreak[1] \\[1mm]
& \hspace*{-35mm} \displaystyle 
\sum_{\text{\shortstack[c]
{$v':B_{\alpha} \cup B_{\beta} \to \{0,1\}$\\
$v' \satc \alpha$
}
}}  \left(\sum_{\ket{u'} \in U_{v'}} \langle \psi{,}u' \rangle^2\right) \qquad (*)
\displaybreak[1] \\[1mm]
& \hspace*{-65mm} \displaystyle 
=  \left(\sum_{\text{\shortstack[c]
{$v:B_{\beta}\to \{0,1\}$\\
$v\satc\beta$
}
}} \sum_{{u} \in U_{v}} \langle \psi{,}u \rangle^2 \right)+ \left(\sum_{\text{\shortstack[c]
{$v: B_{\alpha}\to \{0,1\}$\\
$v\satc\alpha$
}
}} \sum_{{u} \in U_{v}} \langle \psi{,}u \rangle^2\right)  (**)
\displaybreak[1] \\[1mm]
& \hspace*{-55mm} \displaystyle 
=  \Prob_I(\beta) + \Prob_I(\alpha)
\end{align*}
where $(*)$ holds by Proposition~\ref{prop:partition} and $(**)$ holds by the hypothesis $\entc \lneg (\beta \lconj \alpha)$.
\end{proof}

\begin{prop}\em \label{prop:marginproj}
Given a quantum structure  $I=(\cH,\psi, B_j \mapsto \udl{B}_j)$, sets $A' \subseteq A \subseteq B$ such that $\{\udl{B}_j: B_j \in A\}$ is a compatible set, then
 $$\Prob_I(\phi_{A'}^{U'})
		= \sum_{\text{\shortstack[c]
{$U \subseteq A$\\
$U \cap A' = U'$
}
}} \Prob_I(\phi^U_{A})
$$
for every $U' \subseteq A'$.
\end{prop}
\begin{proof}\ \\ 
Assume, without loss of generality that,  $A=\{B_1,\dots,B_n\}$, $A'=\{B_1,\dots,B_k\}$
and $U'=\{B_1,\dots,B_\ell\}$ where $k \leq n$ and $\ell \leq k$.
Observe that:
\begin{align*}  
 \Prob_I(\phi_{A'}^{U'}) =  \sum_{\text{\shortstack[c]
{$v:A' \to \{0,1\}$\\
$v\satc \phi_{A'}^{U'}$
}
}}  \langle \psi,\Proj^{1}_{D^v_{{1}}} \dots  \Proj^{k}_{D^v_{{k}}}\psi\rangle
\displaybreak[1] \\[1mm]
& \hspace*{-64mm} \displaystyle 
= \langle \psi,\Proj^{1}_{\ua} \dots  \Proj^{\ell}_{\ua}\Proj^{{\ell +1}}_{\da} \dots  \Proj^{k}_{\da}\psi\rangle
\displaybreak[1]  \\[2mm]
& \hspace*{-64mm} \displaystyle 
= \langle \psi,\Proj^{1}_{\ua} \dots  \Proj^{\ell}_{\ua}\Proj^{{\ell +1}}_{\da} \dots  
      \Proj^{k}_{\da}
      \displaybreak[1]  \\[0mm]
& \hspace*{-28mm} \displaystyle
       (\Proj^{k+1}_{\ua}+\Proj^{k+1}_{\da})\dots (\Proj^{n}_{\ua}+\Proj^{n}_{\da})\psi\rangle
\displaybreak[1]  \\[1mm]
& \hspace*{-64mm} \displaystyle = \sum_{(D_{k+1},\dots. D_n) \in \{\ua,\da\}^{n-k}}
 \langle \psi,\Proj^{1}_{\ua} \dots  \Proj^{\ell}_{\ua}\Proj^{{\ell +1}}_{\da} \dots  
      \Proj^{k}_{\da}
            \displaybreak[1]  \\[0mm]
& \hspace*{5mm} \displaystyle
       \Proj^{k+1}_{D_{k+1}}\dots \Proj^{n}_{D_n}\psi\rangle
\displaybreak[1]  \\[1mm]
& \hspace*{-64mm} \displaystyle = \sum_{(D_{k+1},\dots. D_n) \in \{\ua,\da\}^{n-k}}
 \Prob_I(\phi^{U' \cup \{D_j: D_j=\ua \text{ and } j=k+1,\dots,n\}}_A)
            \displaybreak[1]  \\[2mm]
& \hspace*{-64mm} \displaystyle = \sum_{\text{\shortstack[c]
{$U \subseteq A$\\
$U \cap A' = U'$
}
}} \Prob_I(\phi^U_{A})\end{align*}
\end{proof}

Now we are ready to show that the proposed axiomatization of $\PLQO$ is strongly sound.

\begin{prop}\em
The logic $\PLQO$ is strongly sound.
\end{prop}
\begin{proof}\ \\
 We only check that axiom $\RR$ is sound since the proof of the others is straightforward. \\[2mm]
Rule ($\RR$). Let $I$ be a quantum structure and $\rho: X \to \reals$ be an assignment. 
Assume that
$$I\rho \sat \beta_1 \lconj \dots \lconj \beta_k$$
and that $$\ds\bigforall
		\left(\left(
		 \bigwedge_{j=1}^k \beta_j^{\RCOF}\right)
					\limp \beta^{\RCOF}\right) \quad (*)$$ is in $\RCOF$. 
Let $\rho': X \to \reals$ be an assignment such that
 $\rho'(x)=\rho(x)$ for $x \in X_\nats$, $$\rho'(x_\alpha)=\Prob_I(\alpha)$$
 for every $\alpha \in \Lc$ such that $\udl{\alpha}$ is a compatible set of observables in $I$,  
 and
 $$\rho'(x_{\{B_i,B_j\}})=\begin{cases}
 0 & \text{if } \{\udl{B_i},\udl{B_j}\} \text{ is a compatible set of observables in } I\\[1mm]
1 & \text{ otherwise}
\end{cases}
$$
for every $B_i,B_j \in B_{(\beta_1 \lconj \dots \lconj \beta_k) \limp \beta}$.
We start by showing that
$$\reals \rho' \satfo \beta_j^\RCOF,$$
for $j=1,\dots,k$.
We consider three cases:\\[1mm]
(1)~$\beta_j$ is $\lO \alpha_j$. Then, $\beta_j^\RCOF$ is $$Q^{\lO \alpha_j}\lconj {Q^{{\inte{}}{\alpha_j}}}.$$
Recall that
$I \rho \sat \lO \alpha_j$. Hence, $\alpha_j$ is $I$-observable, that is
$\udl{\alpha_j}$ is a compatible set of observables in $I$. Thus,
$$\rho'(x_{\{B_{i_1},B_{i_2}\}})=0$$
for every distinct $B_{i_1},B_{i_2} \in B_{\alpha_j}$.
Therefore,
$$\reals \rho' \satfo Q^{{\lO \alpha_j}}.$$
Moreover, ${\reals \rho' \sat Q^{{\inte{}}{\alpha_j}}}$. In fact,  \\[1mm]
(a)~$0 \leq \rho'(x_{\phi_{B_{\alpha_j}}^U}) \leq 1.$ 
Let $U \subseteq B_{\alpha_j}$.  Since
$$\rho'(x_{\phi_{B_{{\alpha_j}}}^U})=\Prob_I(\phi_{B_{{\alpha_j}}}^U)$$
then $0 \leq \rho'(x_{\phi_{B_{{\alpha_j}}}^U}) \leq 1$ by Proposition~\ref{prop:probass}.\\[2mm]
(b)~$\sum_{U \subseteq B_{{\alpha_j}}} \rho'(x_{\phi_{B_{{\alpha_j}}}^U})=1.$
Observe that
$$\bigvee_{U \subseteq B_{{\alpha_j}}} \phi_{B_{{\alpha_j}}}^U$$
is a tautology and so
$$\Prob_I(\bigvee_{U \subseteq B_{{\alpha_j}}} \phi_{B_{{\alpha_j}}}^U)=1$$
by Proposition~\ref{prop:probass}. 
Since, for distinct $U_1,U_2 \subseteq B_{\alpha_j}$, 
$$\entc \lneg(\phi_{B_{{\alpha_j}}}^{U_1} \lconj \phi_{B_{{\alpha_j}}}^{U_2}),$$
then, by Proposition~\ref{prop:probass}, 
$$\sum_{U \subseteq B_{{\alpha_j}}}\Prob_I(\phi_{B_{{\alpha_j}}}^U)=1.$$
Thus, $$\sum_{U \subseteq B_{{\alpha_j}}} \rho'(x_{\phi_{B_{{\alpha_j}}}^U})=1.$$
(c)~$\ds \rho'(x_{\phi_{A'}^{U'} }) = \sum_{\text{\shortstack[c]
{$U \subseteq B_{\alpha_j}$\\
$U \cap A' = U'$
}
}} \rho'(x_{\phi_{B_{{\alpha_j}}}^{U}})$ for every $A' \subseteq B_{\alpha_j}$ and $U' \subseteq A'$.
Observe that
\begin{align*}\rho'(x_{\phi_{A'}^{U'} })= \Prob_I(\phi_{A'}^{U'})
\displaybreak[1] \\[1mm]
& \hspace*{-22mm} \displaystyle 
		= \sum_{\text{\shortstack[c]
{$U \subseteq B_{{\alpha_j}}$\\
$U \cap A' = U'$
}
}} \Prob_I(\phi^U_{B_{\alpha_j}})  \qquad (*)
						\displaybreak[1] \\[1mm]
& \hspace*{-22mm} \displaystyle 
		= \sum_{\text{\shortstack[c]
{$U \subseteq B_{\alpha_j}$\\
$U \cap A' = U'$
}
}} \rho'(x_{\phi_{B_{{\alpha_j}}}^{U}})
						\displaybreak[1] 
\end{align*}
where $(*)$ holds by Proposition~\ref{prop:marginproj}.\\[2mm]
(d)~$\rho'(x_{{\alpha_j}}) = \sum_{\text{\shortstack[c]
{$v:B_{{\alpha_j}} \to \{0,1\}$\\
$v\satc{\alpha_j}$
}
}} \rho'(x_{\phi_{B_{\alpha_j}}^{\{B_j \in B_{\alpha_j}: v(B_j)=1\}}})$. Observe that
$$\entc {\alpha_j} \leqv \left(\bigvee_{\text{\shortstack[c]
{$v:B_{{\alpha_j}} \to \{0,1\}$\\
$v\satc{\alpha_j}$
}
}} \phi_{B_{\alpha_j}}^{\{B_j \in B_{\alpha_j}: v(B_j)=1\}}\right)$$
and that 
$$\entc \lneg(\phi_{B_{\alpha_j}}^{\{B_j \in B_{\alpha_j}: v_1(B_j)=1\}} \lconj \phi_{B_{\alpha_j}}^{\{B_j \in B_{\alpha_j}: v_2(B_j)=1\}})$$
for every distinct pair $v_1,v_2: B_{\alpha_j} \to \{0,1\}$.
Thererefore, by Proposition~\ref{prop:probass}, the thesis follows.
\\[2mm]
(2)~$\beta_j$ is $\inte{\alpha_j} \mycirc p_j$.  Then, $\beta_j^\RCOF$ is $$Q^{\lO \alpha_j} \lconj {Q^{{\inte{}}{\alpha_j}}} \lconj (x_{\alpha_j} \mycirc p).$$
Recall that $I \rho \sat \inte{\alpha_j} \mycirc p_j$. Then
$I \rho \sat \lO \alpha_j$. Thus, by (1), 
$$\reals \rho' \satfo Q^{\lO \alpha_j} \lconj {Q^{{\inte{}}{\alpha_j}}}.$$
So it is enough to prove that
$$\reals \rho' \satfo x_{{\alpha_j}} \mycirc p_j.$$
Since $I \rho \sat \inte{\alpha_j} \mycirc p_j$, then
$$\Prob_I(\alpha_j)\mycirc p_j^\reals.$$
Thus, $$\rho'(x_{\alpha_j})\mycirc p_j^\reals$$
and so, $\reals \rho' \satfo x_{{\alpha_j}} \mycirc p_j$.\\[1mm]
(3)~$\beta_j$ is $\lneg \lO \alpha_j$. 
Then, $\beta_j^\RCOF$ is $$Q^{\lneg \lO \alpha_j}.$$
Recall that $I \rho \sat \lneg \lO \alpha_j$. Hence,
$\alpha_j$ is not an $I$-observable. Then, there are $B_i,B_j \in B_{\alpha_j}$ such that
$\{\udl{B}_i,\udl{B}_j\}$ is an incompatible set of observables in $I$. So, 
$$\rho'(x_{\{B_i,B_j\}})=1.$$
Therefore,
$$\reals \rho' \satfo \lneg Q^{\lO \alpha_j}.$$
Moreover, by definition of $\rho'$,
$$\reals \rho' \satfo \displaystyle  
\left(\bigwedge_{\text{\shortstack[c]
{$A' \subseteq {B_{\alpha_j}}$\\
$|A'|=2$
}
}} x_{A'} \geq 0\right).$$
In this way, we conclude the proof of
$$\reals \rho' \satfo \ds
		\bigwedge_{j=1}^k \beta_j^\RCOF.$$
Returning to the main proof, observe that, by (*), 
$$\reals \rho' \satfo \ds 
		\left(
		\bigwedge_{j=1}^k \beta_j^\RCOF\right)
					\limp \beta^\RCOF,$$
so
$$\reals \rho' \satfo \beta^\RCOF.$$ 
It remains to prove that $I \rho \sat \beta$. There are three cases:\\[1mm]
(1)~$\beta$ is $\lO \alpha$. Hence, $$\reals \rho' \satfo Q^{\lO \alpha}\lconj {Q^{{\inte{}}{\alpha}}}.$$ Therefore,
$\rho'(x_{\{B_{i_1},B_{i_2}\}}) = 0$ for every distinct $B_{i_1},B_{i_2} \in B_{\alpha}$. 
Thus, by definition of $\rho'$, $\{\udl{B_{i_1}},\udl{B_{i_2}}\}$ is a compatible set of quantum variables.
Therefore, $\udl \alpha$ is a compatible set of quantum variables, that is $\alpha$ is an $I$-observable formula
and so $I \rho \sat \lO \alpha$.\\[1mm]
(2)~$\beta$ is $\inte{\alpha} \mycirc p$. Thus, 
$$\reals \rho' \satfo {Q^{\lO \alpha}}\lconj {Q^{{\inte{}}{\alpha}}}  \lconj (x_{{\alpha}} \mycirc p).$$ Therefore, similarly to  (1),
$I \rho \sat \lO \alpha$. Moreover, $\rho'(x_\alpha) \mycirc p^{\reals \rho}$. Hence,
by definition of $\rho'$, $$\Prob_I(\alpha) \mycirc p^{\reals \rho}$$
and so $I \rho \sat \inte{\alpha}\mycirc p$.\\[1mm]
(3)~$\beta$ is $\lneg \lO \alpha$. Thus,
$$\reals \rho' \satfo Q^{\lneg \lO \alpha}.$$
Hence,
$\rho'(x_{\{B_{i_1},B_{i_2}\}}) =1$ for some distinct $B_{i_1},B_{i_2} \in B_{\alpha}$.
Thus, by definition of $\rho'$, $\{\udl{B}_{i_1},\udl{B}_{i_2}\}$ is not a compatible set of quantum variables.
Therefore, $\udl \alpha$ is not a compatible set of quantum variables, that is $\alpha$ is not an $I$-observable formula
and so $I \rho \not \sat \lO \alpha$. Thus, $I \rho \sat \lneg \lO \alpha$.
\end{proof}

\section{Weak completeness} \label{sec:wcomp}

The objective of this section is to show that $\PLQO$ is weakly complete. 
That is, for every formula $\varphi \in L_\PLQO$,
if $\ent \varphi$ then $\der \varphi$. There is no hope that there is a finitary calculus for which $\PLQO$ is strongly complete
since, as we referred above, $\PLQO$ is not compact.

We prove the weak completeness by contrapositive, that is, if $\not \der \varphi$ then $\not \ent \varphi$. Hence, assuming that
$\varphi$ is not derivable, we show that there is a quantum structure that falsifies $\varphi$. With this purpose in mind, we start by defining a generic quantum structure that will be relevant not only for weak completeness but also for conservativeness.

Given a finite set $B' \subseteq B$, $\NC \subseteq B' \times B'$ and a map $f:\{0,\dots, 2^{|B'|}-1\} \to \comps$
such that $\NC$ is irreflexive and symmetric, and 
$$\sqrt{\sum_{j=0}^{2^{|B'|}-1} f(j)\ovl{f(j)}}=1,$$
let
$$I^{B',\NC,f}=(\cH,\psi,B_j \mapsto \udl B_j)$$
be such that
\begin{itemize}
\item $\cH$ is $\comps^{2^{|B'|}+2 |\NC|}$ and fix 
$$\{\zeta_{v_0},\dots,\zeta_{v_{2^{|B'|}-1}}\} \cup \bigcup_{A \in \NC} \{\zeta_{A^\bullet},\zeta_{A_\bullet}\}$$
as a basis for $\cH$ where $v_k: B' \to \{0,1\}$ is the valuation $B_j \mapsto b_j$ where $b_j$ is the $j$-th bit in the binary 
decomposition of $k$ (assuming left padding by zeros);

\item $\psi$ is $$\ds \sum_{j=0}^{2^{|B'|}-1} f(j)\, \zeta_{v_j};$$

\item For each $B_j \in B'$, $\udl{B}_j=(O_j,\{\{0,1\},\comp{\{0,1\}}\},\{0,1\})$ with
\begin{align*}O_j=\left(\sum_{\{v_k: v_k(B_j)=1\}} \zeta_{v_k}\langle \, \zeta_{v_k}\, , \cdot\, \rangle \right) + \left(\sum_{\{v_k: v_k(B_j)=0\}} -\zeta_{v_k}\langle \, \zeta_{v_k}\, , \cdot\, \rangle\right) +
\displaybreak[1] \\[1mm]
& \hspace*{-80mm} \displaystyle  
\left(\sum_{\{(B_i,B_j) \in \NC: j <i\}} \zeta_{{(B_i,B_j)}^\bullet_\bullet} \, \langle \, \zeta_{{(B_i,B_j)}^\bullet_\bullet}\, ,\cdot \, \rangle \right)+
\displaybreak[1] \\[1mm]
& \hspace*{-80mm} \displaystyle  
\left(\sum_{\{(B_i,B_j) \in \NC: j >i\}}  \zeta_{{(B_i,B_j)}_\bullet} \, \langle \, \zeta_{{(B_i,B_j)}_\bullet}\, ,\cdot \, \rangle\right)+
\displaybreak[1] \\[1mm]
& \hspace*{-80mm} \displaystyle  
\left(\sum_{\{A \in \NC: B_ j\notin A\}}  \zeta_{A_\bullet} \, \langle \, \zeta_{A_\bullet}\, ,\cdot \, \rangle +  \zeta_{A^\bullet} \, \langle \, \zeta_{A^\bullet}\, ,\cdot \, \rangle\right) 
\end{align*}  
where
$$\zeta_{{(B_i,B_j)}^\bullet_\bullet}=\sqrt{\frac 12} \, \zeta_{{(B_i,B_j)}^\bullet} + \sqrt{\frac 12} \, \zeta_{{(B_i,B_j)}_\bullet};$$
\item For $B_j \notin B'$, choose for $O_j$ an arbitrary observable on $\cH$, for $\cD_j$ any partition of $\sigma_p(O_j)$ with two elements and for $\ua_j$ one of them.
\end{itemize}

Observe that the chosen basis for the Hilbert space $\cH$ of $I^{B',\NC,f}$ is a finite set including vectors representing the valuations over the finite set $B'$ of propositional symbols and vectors assigned to the non compatible  pairs in $\NC$. The map $f$ corresponds to the ``mass" of each valuation. So, the quantum state $\ket \psi$ is a superposition of the valuation vectors each weighted by its mass. The operator $O_j$ associated with a propositional symbol $B_j \in B'$ is the sum of several operators:
\begin{itemize}

\item Projection $\sum_{\{v_k: v_k(B_j)=1\}} \zeta_{v_k}\langle \, \zeta_{v_k}\, , \cdot\, \rangle$ stating that the vectors associated with the valuations that satisfy $B_j$ are eigenvectors with eigenvalue $1$.

\item Operator $\sum_{\{v_k: v_k(B_j)=0\}}- \zeta_{v_k}\langle \, \zeta_{v_k}\, , \cdot\, \rangle$ stating that the vectors associated with the valuations that do not  satisfy $B_j$ are eigenvectors with eigenvalue $-1$.

\item Projections $\sum_{\{(B_i,B_j) \in \NC: j <i\}} \zeta_{{(B_i,B_j)}^\bullet_\bullet} \, \langle \, \zeta_{{(B_i,B_j)}^\bullet_\bullet}\, ,\cdot \, \rangle$, 
$\sum_{\{(B_i,B_j) \in \NC: j >i\}} \zeta_{{(B_i,B_j)}_\bullet} \, \langle \, \zeta_{{(B_i,B_j)}_\bullet}\, ,\cdot \, \rangle$ referring  to the incompatible pairs where $B_j$ appears which were defined so  that $O_i$ and $O_j$  are not compatible (observe that 
 that the two vectors $\zeta_{{(B_i,B_j)}^\bullet}$ and $\zeta_{{(B_i,B_j)}_\bullet}$ associated with the pair $(B_i,B_j)$ are eigenvectors with eigenvalues $0$ and $1$, respectively). 

\item Projection $\sum_{\{A \in \NC: B_ j\notin A\}}  \zeta_{A_\bullet}\langle \, \zeta_{A_\bullet}\, ,\cdot \, \rangle +  \zeta_{A^\bullet} \, \langle \, \zeta_{A^\bullet}\, ,\cdot \, \rangle$  stating that the vectors referring to the incompatible pairs where $B_j$ does not appear are eigenvectors with eigenvalue $1$. 

\end{itemize}
So, $O_j$ is an observable since it is a diagonizable, bounded and self-adjoint operator. 

The next result states that the pairs of propositional symbols in $\NC$ correspond to incompatible quantum variables in $I^{B',\NC,f}$ and vice-versa. 

\begin{prop}\label{prop:gstructnonc}\em
Let $B_\ell,B_m \in B'$ be distinct variables. 
Then, $\{\udl{B}_\ell,\udl{B}_m\}$ is  not a compatible set of quantum variables in $I^{B',\NC,f}$ iff 
$(B_\ell,B_m)\in \NC$.
\end{prop}
\begin{proof} \ \\
($\from$) 
Assume that $(B_\ell,B_m)\in \NC$ and that  $\ell <m$. Observe that 
\begin{align*}O_\ell=\left(\sum_{\{v_k: v_k(B_\ell)=1\}} \zeta_{v_k}\langle \, \zeta_{v_k}\, , \cdot\, \rangle \right) + \left(\sum_{\{v_k: v_k(B_\ell)=0\}} -\zeta_{v_k}\langle \, \zeta_{v_k}\, , \cdot\, \rangle\right) +\displaybreak[1] \\[1mm]
& \hspace*{-80mm} \displaystyle  
\left(\sum_{\{(B_i,B_\ell) \in \NC: \ell <i, i \neq m\}} 
\zeta_{{(B_i,B_\ell)}^\bullet_\bullet} \, \langle \, \zeta_{{(B_i,B_\ell)}^\bullet_\bullet}\, ,\cdot \, \rangle\right)+
\displaybreak[1] \\[3mm]
& \hspace*{-60mm} \displaystyle \zeta_{{(B_m,B_\ell)}^\bullet_\bullet} \, \langle \, \zeta_{{(B_m,B_\ell)}^\bullet_\bullet}\, ,\cdot \, \rangle\; +
\displaybreak[1] \\[3mm] 
& \hspace*{-80mm} \displaystyle  
\left(\sum_{\{(B_i,B_\ell) \in \NC:\ell >i\}} \zeta_{{(B_i,B_\ell)}_\bullet} \, \langle \, \zeta_{{(B_i,B_\ell)}_\bullet}\, ,\cdot \, \rangle\right)+
\displaybreak[1] \\[1mm]
& \hspace*{-80mm} \displaystyle  
\left(\sum_{\{A \in \NC: B_\ell\notin A\}} \zeta_{A_\bullet} \, \langle \, \zeta_{A_\bullet}\, ,\cdot \, \rangle +  \zeta_{A^\bullet} \, \langle \, \zeta_{A^\bullet}\, ,\cdot \, \rangle\right)
\end{align*}
and
\begin{align*} O_m=\left(\sum_{\{v_k: v_k(B_m)=1\}} \zeta_{v_k}\langle \, \zeta_{v_k}\, , \cdot\, \rangle \right) + \left(\sum_{\{v_k: v_k(B_m)=0\}} -\zeta_{v_k}\langle \, \zeta_{v_k}\, , \cdot\, \rangle\right)  +
\displaybreak[1] \\[1mm]
& \hspace*{-80mm} \displaystyle  
\left(\sum_{\{(B_i,B_m) \in \NC: m <i\}} \zeta_{{(B_i,B_m)}^\bullet_\bullet} \, \langle \, \zeta_{{(B_i,B_m)}^\bullet_\bullet}\, ,\cdot \, \rangle\right)+
\displaybreak[1] \\[3mm]
& \hspace*{-60mm} \displaystyle \zeta_{{(B_m,B_\ell)}_\bullet} \, \langle \, \zeta_{{(B_m,B_\ell)}_\bullet}\, ,\cdot \, \rangle \; +
\displaybreak[1] \\[3mm] 
& \hspace*{-80mm} \displaystyle  
\left(\sum_{\{(B_i,B_m) \in \NC:m >i, i \neq \ell\}} \zeta_{{(B_i,B_m)}_\bullet} \, \langle \, \zeta_{{(B_i,B_m)}_\bullet}\, ,\cdot \, \rangle\right)+
\displaybreak[1] \\[1mm]
& \hspace*{-80mm} \displaystyle  
\left(\sum_{\{A \in \NC: B_m\notin A\}} \zeta_{A_\bullet} \, \langle \, \zeta_{A_\bullet}\, ,\cdot \, \rangle +  \zeta_{A^\bullet} \, \langle \, \zeta_{A^\bullet}\, ,\cdot \, \rangle\right).
\end{align*}
We now show that
$$O_m  O_\ell \neq O_\ell  O_m.$$
Indeed, 
\begin{align*}O_m O_\ell=\left(\sum_{\{v_k: v_k(B_\ell)+v_k(B_m)\neq 1\}} \zeta_{v_k}\langle \, \zeta_{v_k}\, , \cdot\, \rangle\right) +
\displaybreak[1] \\[1mm]
& \hspace*{-62mm} \displaystyle  
\left(\sum_{\{v_k: v_k(B_\ell)+v_k(B_m)=1\}} -\zeta_{v_k}\langle \, \zeta_{v_k}\, , \cdot\, \rangle\right) +
\displaybreak[1] \\[1mm]
& \hspace*{-62mm} \displaystyle  
\left(\sum_{\{(B_i,B_\ell) \in \NC: \ell <i, i \neq m\}} 
\zeta_{{(B_i,B_\ell)}^\bullet_\bullet} \, \langle \, \zeta_{{(B_i,B_\ell)}^\bullet_\bullet}\, ,\cdot \, \rangle\right)+
\displaybreak[1] \\[3mm]
& \hspace*{-50mm} \displaystyle
\sqrt{\frac 12} \; \zeta_{{(B_m,B_\ell)}_\bullet} \, \langle \, \zeta_{{(B_m,B_\ell)}^\bullet_\bullet}\, ,\cdot \, \rangle \; +
\displaybreak[1] \\[3mm] 
& \hspace*{-62mm} \displaystyle  
\left(\sum_{\{(B_i,B_\ell) \in \NC:\ell >i\}} \zeta_{{(B_i,B_\ell)}_\bullet} \, \langle \, \zeta_{{(B_i,B_\ell)}_\bullet}\, ,\cdot \, \rangle\right)+
\displaybreak[1] \\[1mm]
& \hspace*{-62mm} \displaystyle  
\left(\sum_{\{(B_i,B_m) \in \NC: m <i\}} \zeta_{{(B_i,B_m)}^\bullet_\bullet} \, \langle \, \zeta_{{(B_i,B_m)}^\bullet_\bullet}\, ,\cdot \, \rangle\right)+
\displaybreak[1] \\[3mm]
& \hspace*{-62mm} \displaystyle  
\left(\sum_{\{(B_i,B_m) \in \NC:m >i, i \neq \ell\}} \zeta_{{(B_i,B_m)}_\bullet} \, \langle \, \zeta_{{(B_i,B_m)}_\bullet}\, ,\cdot \, \rangle\right)+
\displaybreak[1] \\[1mm]
& \hspace*{-62mm} \displaystyle  
\left(\sum_{\{A \in \NC: B_\ell,B_m\notin A\}} \zeta_{A_\bullet} \, \langle \, \zeta_{A_\bullet}\, ,\cdot \, \rangle +  \zeta_{A^\bullet} \, \langle \, \zeta_{A^\bullet}\, ,\cdot \, \rangle\right).
\end{align*}
On the other hand,
\begin{align*} O_\ell O_m=\left(\sum_{\{v_k: v_k(B_\ell)+v_k(B_m)\neq 1\}} \zeta_{v_k}\langle \, \zeta_{v_k}\, , \cdot\, \rangle\right) +
\displaybreak[1] \\[1mm]
& \hspace*{-62mm} \displaystyle  
\left(\sum_{\{v_k: v_k(B_\ell)+v_k(B_m)=1\}} -\zeta_{v_k}\langle \, \zeta_{v_k}\, , \cdot\, \rangle\right) +
\displaybreak[1] \\[1mm]
& \hspace*{-62mm} \displaystyle  
\left(\sum_{\{\{B_{i},B_m\} \in \NC:m < i\}}  \zeta_{{(B_i,B_m)}^\bullet_\bullet} \, \langle \, \zeta_{{(B_i,B_m)}^\bullet_\bullet}\, ,\cdot \, \rangle\right) +
\displaybreak[1] \\[3mm]
& \hspace*{-50mm} \displaystyle
\sqrt{\frac 12} \; \zeta_{{(B_m,B_\ell)}_\bullet^\bullet} \, \langle \, \zeta_{{(B_m,B_\ell)}_\bullet}\, ,\cdot \, \rangle \;
 \; +
\displaybreak[1] \\[3mm] 
& \hspace*{-62mm} \displaystyle  
\left(\sum_{\{(B_i,B_m) \in \NC:m >i, i \neq \ell\}} \zeta_{{(B_i,B_m)}_\bullet} \, \langle \, \zeta_{{(B_i,B_m)}_\bullet}\, ,\cdot \, \rangle\right)+
\displaybreak[1] \\[1mm]
& \hspace*{-62mm} \displaystyle  
\left(\sum_{\{(B_i,B_\ell) \in \NC: \ell <i, i \neq m\}} \zeta_{{(B_i,B_\ell)}^\bullet_\bullet} \, \langle \, \zeta_{{(B_i,B_\ell)}^\bullet_\bullet}\, ,\cdot \, \rangle\right)+
\displaybreak[1] \\[3mm]
& \hspace*{-62mm} \displaystyle  
\left(\sum_{\{(B_i,B_\ell) \in \NC:\ell >i\}} \zeta_{{(B_i,B_\ell)}_\bullet} \, \langle \, \zeta_{{(B_i,B_\ell)}_\bullet}\, ,\cdot \, \rangle\right)+
\displaybreak[1] \\[1mm]
& \hspace*{-62mm} \displaystyle  
\left(\sum_{\{A \in \NC: B_\ell,B_m\notin A\}} \zeta_{A_\bullet} \, \langle \, \zeta_{A_\bullet}\, ,\cdot \, \rangle +  \zeta_{A^\bullet} \, \langle \, \zeta_{A^\bullet}\, ,\cdot \, \rangle \right).
\end{align*}
Hence
$$O_m O_\ell - O_\ell O_m = {\frac 12} \left (\zeta_{{(B_m,B_\ell)}_\bullet} \, \langle \, \zeta_{{(B_m,B_\ell)}^\bullet}\, ,\cdot \, \rangle- \zeta_{{(B_m,B_\ell)}^\bullet} \, \langle \, \zeta_{{(B_m,B_\ell)}_\bullet}\, ,\cdot \, \rangle\right)\neq 0.$$
Therefore, $\{\udl{B}_m, \udl{B}_\ell\}$ is not a compatible set of quantum variables. \\[2mm]
($\to$) Assume that $(B_\ell,B_m)\notin \NC$. 
Then, 
\begin{align*}O_\ell=\left(\sum_{\{v_k: v_k(B_\ell)=1\}} \zeta_{v_k}\langle \, \zeta_{v_k}\, , \cdot\, \rangle \right) + \left(\sum_{\{v_k: v_k(B_\ell)=0\}} -\zeta_{v_k}\langle \, \zeta_{v_k}\, , \cdot\, \rangle\right) +\displaybreak[1] \\[1mm]
& \hspace*{-80mm} \displaystyle  
\left(\sum_{\{(B_i,B_\ell) \in \NC: \ell <i\}} 
\zeta_{{(B_i,B_\ell)}^\bullet_\bullet} \, \langle \, \zeta_{{(B_i,B_\ell)}^\bullet_\bullet}\, ,\cdot \, \rangle\right)+
\displaybreak[1] \\[3mm]
& \hspace*{-80mm} \displaystyle  
\left(\sum_{\{(B_i,B_\ell) \in \NC:\ell >i\}} \zeta_{{(B_i,B_\ell)}_\bullet} \, \langle \, \zeta_{{(B_i,B_\ell)}_\bullet}\, ,\cdot \, \rangle\right)+
\displaybreak[1] \\[1mm]
& \hspace*{-80mm} \displaystyle  
\left(\sum_{\{A \in \NC: B_\ell\notin A\}} \zeta_{A_\bullet} \, \langle \, \zeta_{A_\bullet}\, ,\cdot \, \rangle +  \zeta_{A^\bullet} \, \langle \, \zeta_{A^\bullet}\, ,\cdot \, \rangle\right)
\end{align*}
and
\begin{align*} O_m=\left(\sum_{\{v_k: v_k(B_m)=1\}} \zeta_{v_k}\langle \, \zeta_{v_k}\, , \cdot\, \rangle \right) + \left(\sum_{\{v_k: v_k(B_m)=0\}} -\zeta_{v_k}\langle \, \zeta_{v_k}\, , \cdot\, \rangle\right)  +
\displaybreak[1] \\[1mm]
& \hspace*{-80mm} \displaystyle  
\left(\sum_{\{(B_i,B_m) \in \NC: m <i\}} \zeta_{{(B_i,B_m)}^\bullet_\bullet} \, \langle \, \zeta_{{(B_i,B_m)}^\bullet_\bullet}\, ,\cdot \, \rangle\right)+
\displaybreak[1] \\[3mm]
& \hspace*{-80mm} \displaystyle  
\left(\sum_{\{(B_i,B_m) \in \NC:m >i\}} \zeta_{{(B_i,B_m)}_\bullet} \, \langle \, \zeta_{{(B_i,B_m)}_\bullet}\, ,\cdot \, \rangle\right)+
\displaybreak[1] \\[1mm]
& \hspace*{-80mm} \displaystyle  
\left(\sum_{\{A \in \NC: B_m\notin A\}} \zeta_{A_\bullet} \, \langle \, \zeta_{A_\bullet}\, ,\cdot \, \rangle +  \zeta_{A^\bullet} \, \langle \, \zeta_{A^\bullet}\, ,\cdot \, \rangle\right).
\end{align*}
Hence, 
\begin{align*}O_m O_\ell=\left(\sum_{\{v_k: v_k(B_\ell)+v_k(B_m)\neq 1\}} \zeta_{v_k}\langle \, \zeta_{v_k}\, , \cdot\, \rangle\right) +
\displaybreak[1] \\[1mm]
& \hspace*{-62mm} \displaystyle  
\left(\sum_{\{v_k: v_k(B_\ell)+v_k(B_m)=1\}} -\zeta_{v_k}\langle \, \zeta_{v_k}\, , \cdot\, \rangle\right) +
\displaybreak[1] \\[1mm]
& \hspace*{-62mm} \displaystyle  
\left(\sum_{\{(B_i,B_m) \in \NC: m <i\}} \zeta_{{(B_i,B_m)}^\bullet_\bullet} \, \langle \, \zeta_{{(B_i,B_m)}^\bullet_\bullet}\, ,\cdot \, \rangle\right)+
\displaybreak[1] \\[3mm]
& \hspace*{-62mm} \displaystyle  
\left(\sum_{\{(B_i,B_m) \in \NC:m >i\}} \zeta_{{(B_i,B_m)}_\bullet} \, \langle \, \zeta_{{(B_i,B_m)}_\bullet}\, ,\cdot \, \rangle\right)+
\displaybreak[1] \\[1mm]
& \hspace*{-62mm} \displaystyle  
\left(\sum_{\{(B_i,B_\ell) \in \NC: \ell <i\}} 
\zeta_{{(B_i,B_\ell)}^\bullet_\bullet} \, \langle \, \zeta_{{(B_i,B_\ell)}^\bullet_\bullet}\, ,\cdot \, \rangle\right)+
\displaybreak[1] \\[3mm]
& \hspace*{-62mm} \displaystyle  
\left(\sum_{\{(B_i,B_\ell) \in \NC:\ell >i\}} \zeta_{{(B_i,B_\ell)}_\bullet} \, \langle \, \zeta_{{(B_i,B_\ell)}_\bullet}\, ,\cdot \, \rangle\right)+
\displaybreak[1] \\[1mm]
& \hspace*{-62mm} \displaystyle  
\left(\sum_{\{A \in \NC: B_\ell,B_m\notin A\}} \zeta_{A_\bullet} \, \langle \, \zeta_{A_\bullet}\, ,\cdot \, \rangle +  \zeta_{A^\bullet} \, \langle \, \zeta_{A^\bullet}\, ,\cdot \, \rangle\right).
\end{align*}
On the other hand,
\begin{align*} O_\ell O_m=\left(\sum_{\{v_k: v_k(B_\ell)+v_k(B_m)\neq 1\}} \zeta_{v_k}\langle \, \zeta_{v_k}\, , \cdot\, \rangle\right) +
\displaybreak[1] \\[1mm]
& \hspace*{-62mm} \displaystyle  
\left(\sum_{\{v_k: v_k(B_\ell)+v_k(B_m)=1\}} -\zeta_{v_k}\langle \, \zeta_{v_k}\, , \cdot\, \rangle\right) +
\displaybreak[1] \\[1mm]
& \hspace*{-62mm} \displaystyle  
\left(\sum_{\{(B_i,B_m) \in \NC: m <i\}} \zeta_{{(B_i,B_m)}^\bullet_\bullet} \, \langle \, \zeta_{{(B_i,B_m)}^\bullet_\bullet}\, ,\cdot \, \rangle\right)+
\displaybreak[1] \\[3mm]
& \hspace*{-62mm} \displaystyle  
\left(\sum_{\{(B_i,B_m) \in \NC:m >i\}} \zeta_{{(B_i,B_m)}_\bullet} \, \langle \, \zeta_{{(B_i,B_m)}_\bullet}\, ,\cdot \, \rangle\right)+
\displaybreak[1] \\[1mm]
& \hspace*{-62mm} \displaystyle  
\left(\sum_{\{(B_i,B_\ell) \in \NC: \ell <i\}} 
\zeta_{{(B_i,B_\ell)}^\bullet_\bullet} \, \langle \, \zeta_{{(B_i,B_\ell)}^\bullet_\bullet}\, ,\cdot \, \rangle\right)+
\displaybreak[1] \\[3mm]
& \hspace*{-62mm} \displaystyle  
\left(\sum_{\{(B_i,B_\ell) \in \NC:\ell >i\}} \zeta_{{(B_i,B_\ell)}_\bullet} \, \langle \, \zeta_{{(B_i,B_\ell)}_\bullet}\, ,\cdot \, \rangle\right)+
\displaybreak[1] \\[1mm]
& \hspace*{-62mm} \displaystyle  
\left(\sum_{\{A \in \NC: B_\ell,B_m\notin A\}} \zeta_{A_\bullet} \, \langle \, \zeta_{A_\bullet}\, ,\cdot \, \rangle +  \zeta_{A^\bullet} \, \langle \, \zeta_{A^\bullet}\, ,\cdot \, \rangle \right).
\end{align*}
Therefore, 
$$O_m  O_\ell -  O_\ell  O_m = 0.$$
Therefore, $\{\udl{B}_m, \udl{B}_\ell\}$ is a compatible set of quantum variables. 
\end{proof}

We now proceed towards the weak completeness of the calculus. We start by proving a finite dimensional model existence lemma showing that we
can move back and forth between satisfaction of certain $\RCOF$ formulas  and satisfaction of $\PLQO$ formulas.

\begin{prop} \em \label{prop:orcftostoval}
Let $\beta_1, \dots,\beta_k$ be formulas of $\PLQO$ of the form $\lO \alpha$, $\lneg\lO \alpha$, $\inte{\alpha}\mycirc p$ and $\rho$ be an assignment over $\reals$ such that 
$$\reals \rho \satfo Q^{\inte{\alpha_j}},$$ for each $j$ such that $\beta_j$ is either of the form $\lO \alpha_j$ or 
$\inte{\alpha_j} \mycirc p_j$.
Then, there is a finite dimensional quantum structure $I$ such that 
$$I\rho \sat \bigwedge_{j=1}^k \beta_j \qquad  \text{iff}\qquad  \reals \rho \satfo \bigwedge_{j=1}^k \beta_j^{\RCOF}.$$
\end{prop}
\begin{proof}  Consider two cases:\\
(1)~There is $j =1,\dots,k$ such $\beta_j$ is either of the form $\lO \alpha_j$ or 
$\inte{\alpha_j} \mycirc p_j$.
Let $B'$ is $\cup_{j=1}^k B_{\beta_j}$. 
Assume without loss of generality that
$B'$ is the set $\{B_1,\dots, B_{|B'|}\}$.
Consider the quantum structure
$$I=I^{B',\NC,f}$$
where
$$\NC=\{(B_i,B_j): \rho(x_{\{B_i,B_j\}}) \neq 0, B_i, B_j \in B', i \neq j\}$$
and $$f(j)=  \sqrt{\rho(x_{\phi_{B'}^{U_j}})}$$
where $U_j=\{B_i \in B': v_j(B_i)=1\}$
for each $j= 0,\dots,{2^{|B'|}-1}$ and each valuation $v_j$ is as defined in $I^{B',\NC,f}$.
\\[2mm]
We now prove that $$I\rho \sat \beta_j \quad  \text{iff}\quad  \reals \rho \satfo \beta_j^{\RCOF}$$
for each $j=1,\dots,k$.\\[2mm]
There are three cases to consider:\\[1mm]
(a)~$\beta_j$ is $\lO \alpha$. Recall that $(\lO \alpha)^{\RCOF}$ is 
$Q^{\lO \alpha} \lconj {Q^{{\inte{}}{\alpha}}}$.\\[1mm]
$(\from)$ Assume that $\reals \rho \satfo Q^{\lO \alpha} \lconj {Q^{{\inte{}}{\alpha}}}$. 
 Hence, $\reals \rho \satfo x_{A'} =0$ for every $A' \subseteq B_{\alpha}$ with $|A'|=2$ since
$\reals \rho \satfo Q^{\lO \alpha}$.  
We must show that
$\udl{\alpha}=\{\udl{B}_j: B_j \in B_{\alpha}\}$ is a compatible set of quantum variables.  Let
$\{\udl{B}_i,\udl{B}_j\}$ be a set of quantum variables  in $\udl{\alpha}$. Observe that $(B_i,B_j) \notin \NC$.
Thus,  by Proposition~\ref{prop:gstructnonc}, $\{\udl{B}_i,\udl{B}_j\}$ is a compatible set of quantum variables.
Thus, $I \rho \sat \lO \alpha$.\\[2mm]
$(\to)$ Assume that $\reals \rho \nsatfo (\lO \alpha)^{\RCOF}$. Then, 
$\reals \rho \nsatfo Q^{\lO \alpha}$ since, by hypothesis, $$\reals \rho \satfo Q^{\inte{\alpha}}.$$ 
Then, there is $\{B_\ell,B_m\} \subseteq B_{\alpha}$ such that
$$\rho(x_{\{B_\ell,B_m\}}) \neq 0.$$
Hence, $(B_\ell,B_m) \in \NC$. Thus,  by Proposition~\ref{prop:gstructnonc},  $\{\udl{B}_m, \udl{B}_\ell\}$ is not a compatible set of quantum variables. Therefore, also $\udl{\alpha}$ 
is not a compatible set of quantum variables and so $\alpha$ is not an $I$-observable formula. Thus,
$I \rho \not \sat \lO \alpha$.
\\[2mm]
(b)~$\beta_j$ is $\inte{\alpha} \mycirc p$. Recall that $(\inte{\alpha} \mycirc p)^{\RCOF}$ is $Q^{\lO \alpha} \lconj {Q^{{\inte{}}{\alpha}}} \lconj (x_{\alpha} \mycirc p)$. Observe that
\begin{align*}  
 \Prob_I(\alpha) =  \sum_{\text{\shortstack[c]
{$v: B_{\alpha} \to \{0,1\}$\\
$v\satc\alpha$
}
}} \left(\sum_{{u} \in U_v} \langle \psi{,}u \rangle^2\right) \qquad  \qquad (*)
					\displaybreak[1] \\[1mm]
& \hspace*{-76mm} \displaystyle 
= \sum_{\text{\shortstack[c]
{$v: B_{\alpha} \to \{0,1\}$\\
$v\satc\alpha$
}
}}  \left(\sum_{\text{\shortstack[c]
{$v': B' \to \{0,1\}$\\
$v'|_{B_\alpha}= v$
}
}} \langle \psi{,}\, \zeta_{v'} \rangle^2\right) \qquad  \qquad (**)
\displaybreak[1] \\[1mm]
& \hspace*{-76mm} \displaystyle 
= \sum_{\text{\shortstack[c]
{$v: B_{\alpha} \to \{0,1\}$\\
$v\satc\alpha$
}
}}  \left(\sum_{\text{\shortstack[c]
{$v': B' \to \{0,1\}$\\
$v'|_{B_\alpha}= v$
}
}} \rho(x_{\phi_{B'}^{\{B_k \in B': v'(B_k)=1\}}})\right) 
\displaybreak[1]\\[1mm]
& \hspace*{-76mm} \displaystyle 
= \sum_{\text{\shortstack[c]
{$v: B_{\alpha} \to \{0,1\}$\\
$v\satc\alpha$
}
}} \rho(x_{\phi^{\{B_j \in B_{\alpha}: v(B_j)=1\}}_{B_{\alpha}}})
\displaybreak[1]\\[1mm]
& \hspace*{-76mm} \displaystyle 
= \rho(x_{\alpha})
\end{align*}
since $\reals \rho \satfo {Q^{{\inte{}}{\alpha}}}$,  (*) holds by Proposition~\ref{prop:probIatr} and (**)  holds since 
$$U_v \cap \{\zeta_{v_0},\dots,\zeta_{v_{2^{|B'|}-1}}\} =\{v': v': B' \to \{0,1\} \text{ and }
v'|_{B_\alpha}=v\}$$ and 
 $$\langle \psi{,}u \rangle =0$$ for each 
 ${u} \in \bigcup_{A \in \NC} \{\zeta_{A^\bullet},\zeta_{A_\bullet}\}$.\\[3mm]
$(\from)$ Assume that $\reals \rho \satfo \beta_j^{\RCOF}$. 
Therefore,
$$\reals \rho \satfo Q^{\lO \alpha} \lconj {Q^{{\inte{}}{\alpha}}} \lconj (x_{\alpha} \mycirc p)$$ and following a similar reasoning to the one in (a), we can conclude
that $I \rho \sat \lO \alpha$. Moreover,
$\rho(x_{\alpha})\mycirc p^{\reals \rho}$.
 Then, 
 $$\Prob_I(\alpha) 
= \rho(x_{\alpha})
\mycirc p^{\reals \rho}.$$
So $I\rho \sat \beta_j$.\\[2mm]
$(\to)$ Assume that $\reals \rho \nsatfo \beta_j^{\RCOF}$. Since, by hypothesis, $\reals \rho \satfo {Q^{{\inte{}}{\alpha}}}$, there are only two cases to consider.\\[1mm]
(i)~$\reals \rho \nsatfo Q^{\lO \alpha}$. Then, by (a), $I\rho \not \sat \lO \alpha$ and so
$I \rho \not \sat \beta_j$.\\[1mm]
(ii)~$\reals \rho \nsatfo x_{\alpha}\mycirc p$. 
Hence,
$$\Prob_I(\alpha) 
= \rho(x_{\alpha})
\nmycirc p^{\reals \rho}.$$
Therefore, $I \rho \not \sat \beta_j$.\\[2mm]
(c)~$\beta_j$ is $\lneg \lO \alpha$. Recall that $(\lneg \lO \alpha)^{\RCOF}$ is 
$\lneg Q^{\lO \alpha}$.
Hence, 
 $$
\begin{array}{cl}
\reals \rho \satfo (\lneg \lO \alpha)^{\RCOF} \\[1mm]
      \text{iff} & \\[1mm]
     \reals \rho \satfo \lneg Q^{\lO \alpha}\\[1mm]
           \text{iff} & \\[1mm]
                \reals \rho \satfo \lneg (Q^{\lO \alpha} \lconj   {Q^{{\inte{}}{\alpha}}})\\[1mm]
           \text{iff} & \\[1mm]
                \reals \rho \satfo \lneg (\lO \alpha)^{\RCOF}\\[1mm]
                           \text{iff} & \\[1mm]
                \reals \rho \nsatfo  (\lO \alpha)^{\RCOF}\\[1mm]

           \text{iff} & \\[1mm]
         I\rho  \not \sat \lO \alpha \\[1mm]
      \text{iff} & \\[1mm]
      I\rho \sat \lneg \lO \alpha.
\end{array}
$$
Therefore,  
$$
\begin{array}{cl}
\reals \rho \satfo \beta_1 \lconj \dots \lconj \beta_{k}  \\[1mm]
      \text{iff} & \\[1mm]
     \reals \rho \satfo \beta_1 \quad \dots \quad \reals \rho\satfo \beta_{k}\\[1mm]
           \text{iff} & \\[1mm]
         I\rho  \sat \beta_1  \quad \dots \quad I \rho\satfo \beta_{k}\\[1mm]
      \text{iff} & \\[1mm]
      I\rho \sat \beta_1 \lconj \dots \lconj \beta_{k}
\end{array}
$$
as we wanted to show.
\end{proof}

\begin{pth} \em
The logic $\PLQO$ is weakly complete.
\end{pth}
\begin{proof}\ \\
Let $\varphi \in L_\PLQO$. 
Assume that $\not \der \varphi$. We proceed to show that $\not \ent \varphi$.
First observe that $\lneg\varphi$ must be consistent 
because otherwise from $\lneg\varphi$ one would be able to derive every formula, including in particular, 
$ \varphi$, and so $\der (\lneg \varphi) \limp \varphi$, and in that case one would have $\der\varphi$. 
Observe that there are conjunctions  $\eta_1,\dots,\eta_\ell$ of literals of the form $\lO \alpha$, $\lneg\lO \alpha$ and $\inte{\alpha}\mycirc p$ such that 
$$\der (\lneg \varphi) \displaystyle \leqv \bigvee_{i=1}^\ell \eta_i.$$ Since $\lneg\varphi$ is consistent, at least one of the disjuncts must also be consistent.
Let $\eta_m$ be one such consistent disjunct. Assume that $\eta_m$ is $$\beta_1\lconj \dots \lconj\beta_k$$
where each $\beta_j$ is a literal.
For proving that $\not \ent \varphi$ it is enough to show that $\lneg\varphi$ is satisfiable. 
Hence, it is enough to show that there is one satisfiable disjunct. Indeed, $\eta_m$ is satisfiable as we proceed to prove.
Towards a contradiction, assume that there are no $I$ and $\rho$ such that 
$I \rho \sat \eta_m$ holds.
Then, by Proposition~\ref{prop:orcftostoval}, there would not exist $\rho$ such that
$$\reals \rho \satfo   \bigwedge_{j: \,\beta_j \text{is either} \lO\alpha_j \text{or}\inte{}\! \alpha_j\mycirc p_j }Q^{\inte{}\alpha_j} \; \lconj \; \bigwedge_{j=1}^k \beta_j^{\RCOF}.$$
So, there would not exist $\rho$ such that
$$\reals \rho \satfo   \bigwedge_{j=1}^k \beta_j^{\RCOF}.$$
Hence, we would have 
$$\ds\bigforall
		\left(\left(\bigwedge_{j=1}^k \beta_j^{\RCOF}\right)
					\limp \lfalsum^\RCOF \right)\in \\[1mm] \hspace*\fill \RCOF.$$
Then, by $\RR$, we would establish $$\eta_m \der \lfalsum$$ 
in contradiction with the consistency of $\eta_m$.
\end{proof}

\section{Conservativeness}\label{sec:cons}

In this section, we show that $\PLQO$ is a conservative extension of classical propositional logic modulo a translation
of classical propositional formulas  into the language of $\PLQO$.

\begin{prop} \em
Let $\alpha \in \Lc$. Then
$$\lO \alpha \ent \inte{\alpha} =1 \quad \quad \text{ iff } \quad \quad \entc \alpha.$$
\end{prop}
\begin{proof}\ \\
($\to$) Assume that $\lO \alpha \ent \inte{\alpha} =1$. 
Consider the quantum structure
$$I= I^{B_{\alpha},\emptyset,f}$$
where $f(j)= \sqrt{\frac{1}{2^{|B_{\alpha}|}}}$ for every $j=0,\dots, 2^{|B_{\alpha}|}-1$, and an arbitrary assignment $\rho$ over
$\reals$.
Thus,  by Proposition~\ref{prop:gstructnonc},  $\{\udl{B}_m, \udl{B}_\ell\}$ is a compatible set of quantum variables for $I$
for every $\udl{B}_m, \udl{B}_\ell \in \udl{\alpha}$. Therefore,
$$I \rho \sat \lO \alpha.$$ 
Hence, by the hypothesis, 
 $I \rho \sat \inte{\alpha} =1$. Then, $\Prob_I(\alpha)=1$. Observe that, by Proposition~\ref{prop:probIatr},
$$\Prob_I(\alpha)= \sum_{\text{\shortstack[c]
{$v: B_{\alpha} \to \{0,1\}$\\
$v\satc\alpha$
}
}} \left(\sum_{{u} \in U_v} \langle \psi{,}u \rangle^2\right)$$
and that
$$\sum_{{u} \in U_v} \langle \psi{,}u \rangle^2= \frac{1}{2^{|B_{\alpha}|}},$$
for each $v:B_\alpha \to \{0,1\}$, because $U_v=\{\zeta_{v}\}$.
Thus, since 
$$\sum_{\text{\shortstack[c]
{$v: B_{\alpha} \to \{0,1\}$\\
$v\satc\alpha$
}
}} \left(\sum_{{u} \in U_v} \langle \psi{,}u \rangle^2\right)=1$$
we have that $\{v: v:B_\alpha \to \{0,1\}\}=\{v: v:B_\alpha \to \{0,1\}, v \satc \alpha\}$.
Therefore, $\alpha$ is a tautology.
\\[2mm]
($\from$) Assume that $\entc\alpha$. Let $I$ be a quantum structure and $\rho$ an arbitrary assignment over $\reals$.
Assume $I \rho \sat \lO \alpha$. Since $\Prob_I$ is a probability assignment by Proposition~\ref{prop:probass},
$\Prob_I(\alpha)=1$ and so $I \rho \sat \inte{\alpha} =1$.
\end{proof}

Then, writing 
$$\text{$\alpha^*$ \; for \; $(\lO \alpha) \limp (\inte{\alpha} =1)$},$$
for the translation of classical formula $\alpha$ into the language of $\PLQO$,
the following corollary holds.

\begin{prop} \em
Let $\alpha \in \Lc$. Then
$$\ent \alpha^* \quad \text{ iff } \quad \entc \alpha.$$
\end{prop}


\section{Epistemic nature of observability}\label{sec:epistemic}

We now analyse the epistemic aspects of quantum reasoning of observability in $\PLQO$ following the steps in~\cite{acs:15}. Recall that $\lO\alpha$ intuitively means that it is possible to do a simultaneous measurement over the essential propositional symbols of $\alpha$ in order to know its  truth value. With this  in mind, we discuss the behaviour of $\lO$ with respect to the propositional connectives. 

We start by considering the law 
$$((\lO\alpha_1) \lconj (\lO\alpha_2))  \leqv (\lO(\alpha_1 \lconj \alpha_2))$$
which is characteristic of an epistemic operator (see for instance~\cite{mey:95}). Observe that
there are classical formulas $\alpha_1$ and $\alpha_2$ such that 
$$\not \ent ((\lO\alpha_1) \lconj (\lO\alpha_2))  \leqv (\lO(\alpha_1 \lconj \alpha_2)).$$
Indeed, consider the quantum structure $I_1$ defined in Example~\ref{ex:1} and let $\alpha_1$ be $B_1$ and $\alpha_2$ be $B_2$.

\begin{example}\em \label{ex:1}
Recall Example~\ref{ex:spin1}. 
Let
$$I=(\comps^2, \psi, B_j \mapsto \udl{B}_j)$$
be the quantum structure 
where 
\begin{itemize}

\item $\udl{B}_1$ is a propositional quantum variable such that $O_1$ is the observable $O_x$ for the $x$ component of a spin-$\frac 12$ particle;

\item $\udl{B}_2$ is a propositional quantum variable such that $O_2$ is the observable $O_y$ for the $y$ component of a spin-$\frac 12$ particle;

\item $\udl{B}_3$ is a propositional quantum variable such that $O_3$ is the observable $O^2$ for the total angular momentum of a spin-$\frac 12$ particle.
\end{itemize}
Hence,  
$\udl{B}_1$ and $\udl{B}_2$ are not compatible. Therefore, 
 $$I \not \sat \lO (B_1 \lconj B_2).$$ 
 On the other hand,
 $I  \sat \lO (B_1)$ and $I_1  \sat \lO (B_2)$ and so 
 $$I \not \sat ((\lO B_1) \lconj (\lO B_2))  \leqv (\lO(B_1 \lconj B_2)).$$
However, 
$$I \sat ((\lO B_1) \lconj (\lO B_3))  \leqv (\lO(B_1 \lconj B_3))$$
since $\udl{B}_1$ and $\udl{B}_3$ are compatible.

Finally, there is no surprise at all due to the intended meaning of $\lO$  that the following
distributive law 
$$(\lO (\alpha_1 \lconj (\alpha_2 \ldisj \alpha_3))) \leqv ((\lO(\alpha_1 \lconj \alpha_2) \ldisj (\lO(\alpha_1 \lconj \alpha_3)))$$
does not hold. 
In fact, 
$$I \not \sat (\lO (B_3 \lconj (B_1 \ldisj B_2))) \leqv ((\lO(B_3 \lconj B_1) \ldisj (\lO(B_3 \lconj B_2)))$$
holds. \curtains
\end{example}

In the same vein, we can show that 
$$((\lO\alpha_1) \ldisj (\lO\alpha_2))  \leqv (\lO(\alpha_1 \ldisj \alpha_2))$$
is satisfiable for some classical formulas $\alpha_1$ and $\alpha_2$ but does not hold in general, and the same for the law 
$$((\lO\alpha_1) \limp (\lO\alpha_2))  \limp (\lO(\alpha_1 \limp \alpha_2)).$$

Concerning other laws of epistemic logic, note that the introspection axioms cannot be considered in our setting since the language of 
$\PLQO$  does not allow formulas  with nested $\lO$ operators. On the other hand, the  knowledge axiom 
$$(\lO \alpha) \limp \alpha^*$$
holds 
in general for any tautology $\alpha$.

Finally, observe that the formula 
$$(\lO \alpha) \leqv (\lO (\lneg \alpha))$$
holds in general since 
$\eB_\alpha = \eB_{\lneg \alpha}$. Indeed, 
assume that $B_j \in \eB_\alpha$. Then, for every valuation $v$ on $B_\alpha$, 
$$v^{B_j \mapsfrom 0} \satc \alpha \quad \text{ iff } \quad v^{B_j \mapsfrom 1} \not \satc \alpha.$$
Hence,
$$v^{B_j \mapsfrom 0} \satc \lneg \alpha \quad \text{ iff } \quad v^{B_j \mapsfrom 1} \not \satc \lneg \alpha.$$
taking into account $v$ is also a valuation on $B_{\lneg \alpha}$. So, $B_j \in \eB_{\lneg \alpha}$.
Similarly, for the other inclusion.

A similar epistemic behaviour occurs for example for the operator $\lM$ defined
as an abbreviation of $\inte{\alpha}>0$. The intuitive meaning of 
$\lM \alpha$ is that $\alpha$ is observable and possibly true.
Since, $\lM\alpha$ can be seen as a conjunction involving $\lO \alpha$,  then, as a consequence of the previous 
discussion, it is immediate to extrapolate its behaviour  with respect to  the propositional connectives. 
 On the other hand,
 $$\not \ent (\lM \alpha) \leqv (\lM (\lneg \alpha))$$
when $\alpha$ is a tautology. In fact, when $\alpha$ is a tautology, $\lM \alpha$ holds but $\lM (\lneg \alpha)$ does not hold
 since although $\lO (\lneg \alpha)$ holds we have $\inte{\lneg \alpha}=0$. However,
 $\ent (\lM \alpha) \leqv (\lM (\lneg \alpha))$  when both $\alpha$ and $\lneg \alpha$ are not tautologies.

\section{Concluding remarks}\label{sec:concrems}

Probabilistic reasoning about quantum observability was introduced into classical propositional logic by developing the logic $\PLQO$. 
 The semantics of $\PLQO$ was given by quantum structures each one composed by a Hilbert space representing the quantum system together with a unit vector representing the current state plus the interpretation of each propositional symbol as a  propositional quantum variable over a diagonizable and bounded observable. The logic was axiomatized by relying on the decidable theory of real closed ordered fields and shown to be  strongly sound and weakly complete. Furthermore, $\PLQO$ was proved to be a conservative extension of classical propositional logic (modulo a suitable translation). The epistemic nature of the observability operator $\lO$ was discussed. Although $\lO$ satisfies at least in part some epistemic properties, it is of a different nature of other quantum-like epistemic operators, like for instance those in~\cite{sme:10}, defined for different purposes. 
 
Concerning future work, we intend to extend the logic to cope with unbounded operators. Moreover,
we intend to investigate the consequence of discarding the assumption that simultaneous measurements imply compatibility of observables.

Furthermore, it  seems worthwhile to study  
other meta-properties of $\PLQO$ starting with decidability and complexity
of the decision problem. We expect this complexity to be much lower than the complexity of $\RCOF$ theoremhood since we only need to recognize $\RCOF$ theorems of a very simple clausal form. 

Strong completeness of the $\PLQO$ axiomatization is not possible
because $\reals$ is an Archimedean real closed ordered field which led to a non compact entailment. Relaxing the semantics may open the door to establishing strong completeness.  


Finally, it seems worthwhile to analyze the observability operator $\lO$ as an analog of the consistency operator in some 
paraconsistent and paracomplete logics~\cite{car:07}. 

\section*{Acknowledgments}
 
This work was supported by Fundação para a Ciência e a Tecnologia by way of grant UID/MAT/04561/2013 to 
Centro de Matemática, Aplicações Fundamentais e Investigação Operacional of Universidade de Lisboa (CMAF-CIO), grant  
UID/FIS/00099/2013 to Centro Multidisciplinar de Astrofísica and grant UID/EEA/50008/2013 to Instituto de Telecomunicações.




\begin{thebibliography}{MvdH95}

\bibitem[Ada98]{adam:98}
E.~W. Adams.
\newblock {\em A Primer of Probability Logic}.
\newblock {CSLI}, 1998.

\bibitem[BEH08]{bla:08}
J.~Blank, P.~Exner, and M.~Havl\'\i{\v c}ek.
\newblock {\em Hilbert Space Operators in Quantum Physics}.
\newblock Theoretical and Mathematical Physics. Springer, New York; AIP Press,
  New York, second edition, 2008.

\bibitem[BS10]{sme:10}
A.~Baltag and S.~Smets.
\newblock Correlated knowledge: an epistemic-logic view on quantum
  entanglement.
\newblock {\em International Journal of Theoretical Physics},
  49(12):3005--3021, 2010.

\bibitem[BvN36]{neu:36}
G.~Birkhoff and J.~von Neumann.
\newblock The logic of quantum mechanics.
\newblock {\em Annals of Mathematics}, 37(4):823--843, 1936.

\bibitem[CCM07]{car:07}
W.~Carnielli, M.~E. Coniglio, and J.~Marcos.
\newblock Logics of formal inconsistency.
\newblock In D.~Gabbay and F.~Guenthner, editors, {\em Handbook of
  Philosophical Logic}, volume~14, pages 1--93. Springer, 2007.

\bibitem[DCG02]{chi:02}
M.~L. Dalla~Chiara and R.~Giuntini.
\newblock Quantum logics.
\newblock In D.~Gabbay and F.~Guenthner, editors, {\em Handbook of
  Philosophical Logic}, volume~6, pages 129--228. Springer, 2002.

\bibitem[Dir67]{dir:67}
P.~A.~M. Dirac.
\newblock {\em The Principles of Quantum Mechanics}.
\newblock Oxford University Press, (revised) fourth edition, 1967.

\bibitem[EGL07]{gab:07}
K.~Engesser, D.~Gabbay, and D.~Lehmann, editors.
\newblock {\em Handbook of Quantum Logic and Quantum Structures: Quantum
  Structures}.
\newblock Elsevier, 2007.

\bibitem[EGL09]{gab:09}
K.~Engesser, D.~Gabbay, and D.~Lehmann, editors.
\newblock {\em Handbook of Quantum Logic and Quantum Structures: Quantum
  Logic}.
\newblock Elsevier, 2009.

\bibitem[Gri14]{gri:14}
R.~B. Griffiths.
\newblock The new quantum logic.
\newblock {\em Foundations of Physics}, 44(6):610--640, 2014.

\bibitem[Hal03]{hal:03}
J.~Y. Halpern.
\newblock {\em Reasoning about Uncertainty}.
\newblock MIT Press, 2003.

\bibitem[Hal13]{hal:13}
B.~Hall.
\newblock {\em Quantum Theory for Mathematicians}.
\newblock Springer, Graduate Texts in Mathematics, 2013.

\bibitem[Har81]{har:81}
G.~M. Hardegree.
\newblock An axiom system for orthomodular quantum logic.
\newblock {\em Studia Logica}, 40(1):1--12, 1981.

\bibitem[Har16]{har:16}
C.~Hartonas.
\newblock First-order frames for orthomodular quantum logic.
\newblock {\em Journal of Applied Non-Classical Logics}, 26(1):69--80, 2016.

\bibitem[HS15]{acs:15}
A.~B. Henriques and A.~Sernadas.
\newblock Epistemic nature of quantum reasoning.
\newblock Preprint, IST - U Lisboa, 1049-001 Lisboa, Portugal, 2015.

\bibitem[Ish95]{ish:95}
C.~J. Isham.
\newblock {\em Lectures on Quantum Theory: Mathematical and Structural
  Foundations}.
\newblock Imperial College Press, 1995.

\bibitem[Mar02]{mar:02}
D.~Marker.
\newblock {\em Model Theory: {A}n {I}ntroduction}, volume 217 of {\em Graduate
  Texts in Mathematics}.
\newblock Springer-Verlag, 2002.

\bibitem[MS06]{pmat:acs:05}
P.~Mateus and A.~Sernadas.
\newblock Weakly complete axiomatization of exogenous quantum propositional
  logic.
\newblock {\em Information and Computation}, 204(5):771--794, 2006.

\bibitem[MvdH95]{mey:95}
J.-J. Meyer and W.~van~der Hoeck.
\newblock {\em Epistemic Logic for AI and Computer Science}.
\newblock Cambridge University Press, 1995.

\bibitem[Pav16]{pav:16}
M.~Pavi{\v{c}}{i\'c}.
\newblock Classical logic and quantum logic with multiple and common lattice
  models.
\newblock {\em Advances in Mathematical Physics}, pages Art. ID 6830685, 12,
  2016.

\bibitem[PM08]{pav:08}
M.~Pavi{\v{c}}{i\'c} and N.~D. Megill.
\newblock Standard logics are valuation-nonmonotonic.
\newblock {\em Journal of Logic and Computation}, 18(6):959--982, 2008.

\bibitem[Rud87]{rud:87}
W.~Rudin.
\newblock {\em Real and Complex Analysis}.
\newblock McGraw-Hill Book Co., New York, third edition, 1987.

\bibitem[Rud91]{rud:91}
W.~Rudin.
\newblock {\em Functional Analysis}.
\newblock McGraw-Hill, Inc., New York, second edition, 1991.

\bibitem[SP07]{ste:pos:07}
B.~Steinbach and C.~Posthoff.
\newblock An extended theory of boolean normal forms.
\newblock In {\em Proceedings of the 6th Annual Hawaii International Conference
  on Statistics, Mathematics and Related Fields}, pages 1124--1139, 2007.

\bibitem[SRS17]{acs:jfr:css:15}
A.~Sernadas, J.~Rasga, and C.~Sernadas.
\newblock On probability and logic.
\newblock {\em Portugaliae Mathematica}, 2017.
\newblock In print.

\bibitem[Tak81]{tak:81}
G.~Takeuti.
\newblock Quantum set theory.
\newblock In {\em Current Issues in Quantum Logic}, volume~8 of {\em Ettore
  Majorana Internat. Sci. Ser.: Phys. Sci.}, pages 303--322. Plenum, New
  York-London, 1981.

\bibitem[Tar51]{tar:51}
A.~Tarski.
\newblock {\em A Decision Method for Elementary Algebra and Geometry}.
\newblock University of California Press, 1951.
\newblock 2nd Edition.

\bibitem[vdMP03]{mey:03}
R.~van~der Meyden and M.~K. Patra.
\newblock A logic for probability in quantum systems.
\newblock In {\em Computer Science Logic, 17th International Workshop, 12th
  Annual Conference of the EACSL, and 8th Kurt G{\"{o}}del Colloquium}, volume
  2803, pages 427--440. Springer, 2003.

\bibitem[vN55]{neu:32}
J.~von Neumann.
\newblock {\em Mathematical Foundations of Quantum Mechanics}.
\newblock Princeton Landmarks in Mathematics. Princeton University Press, 1955.
\newblock Translated from the German (1932).

\bibitem[Zhe27]{zhe:27}
I.~I. Zhegalkin.
\newblock On the technique of calculating propositions in symbolic logic.
\newblock {\em Matematicheskii Sbornik}, 1927.

\end{thebibliography}

\end{document}

